\documentclass[11pt, leqno]{amsart}
\usepackage{amsmath,amsthm,amsfonts,amssymb,amscd,tikz}

\renewcommand{\L}{\mathbb{L}}
\newcommand{\R}{\mathbb{R}}

\newcommand{\cZ}{\mathcal{Z}}

\newcommand{\bbP}{\mathbb{P}}

\def\bP{{\mathbf P}}

\newcommand{\cX}{{\mathcal X}}
\newcommand{\cY}{{\mathcal Y}}

\newcommand{\cO}{{\mathcal O}}

\newcommand{\cL}{{\mathcal L}}

\newcommand{\cA}{{\mathcal A}}
\newcommand{\cB}{{\mathcal B}}

\newcommand{\cU}{{\mathcal U}}

\newcommand{\cH}{{\mathcal H}}

\newcommand{\Hom}{\operatorname{Hom}}

\newcommand{\Sym}{\mathrm{Sym}}

\newcommand{\End}{\mathrm{End}}

\newcommand{\U}{\operatorname{U}}

\newcommand{\Perf}{\operatorname{Perf}}

\newcommand{\Hilb}{\rm Hilb}

\newcommand{\id}{\operatorname{id}}

\newcommand{\Pic}{\operatorname{Pic}}

\newcommand{\bbC}{{\mathbb C}}

\newcommand{\bbR}{{\mathbb R}}

\newcommand{\bbZ}{{\mathbb Z}}

\renewcommand{\bbP}{{\mathbb P}}

\newcommand{\bbQ}{{\mathbb Q}}

\newtheorem{thm}{Theorem}[section]

\newtheorem{theorem}[thm]{Theorem}
\newtheorem{cor}[thm]{Corollary}

\newtheorem{prop}[thm]{Proposition}
\newtheorem{lemma}[thm]{Lemma}

\newtheorem{conj}[thm]{Conjecture}

\newtheorem{remark}[thm]{Remark}

\newtheorem{defn}[thm]{Definition}

\newtheorem*{claim*}{Claim}
\numberwithin{equation}{section}

\begin{document}

\begin{abstract} Let $X$ be a smooth complex projective variety. The group of autoequivalences of the derived category of $X$ acts naturally on its singular cohomology $H^\bullet (X,\bbQ)$ and we denote by $G^{eq}(X)\subset Gl(H^\bullet (X,\bbQ))$ its image. Let $\overline{G^{eq}(X)}\subset Gl(H^\bullet (X,\bbQ)$ be its Zariski closure. We study the relation of the Lie algebra $Lie\overline{G^{eq}(X)}$ and the Neron-Severi Lie algebra $\mathfrak{g}_{NS}(X)\subset End (H(X,\bbQ))$ in case $X$ has trivial canonical line bundle.

At the same time for mirror symmetric families of (weakly) Calabi-Yau varieties we consider a conjecture of Kontsevich on the relation between the monodromy of one family and the group $G^{eq}(X)$ for a very general member $X$ of the other family.
\end{abstract}

\title{Neron-Severi Lie algebra, autoequivalences of the derived category, and monodromy}

\author{Valery A. Lunts}
\address{Department of Mathematics,
Indiana University,
Bloomington, IN
47405,}

\address{National Research University Higher School of Economics, Moscow, Russia}

\thanks{The author was supported by the Basic Research Program of the National Research 
University Higher School of Economics}

\maketitle

\section{Introduction}

\subsection{Lie algebra $\mathfrak{g}_{NS}(X)$ and the group $G^{eq}(X)$} Let $X$ be a smooth complex projective variety of dimension $n$. Consider the semi-simple operator $h\in End(H^\bullet (X,\bbQ))$ which acts as multiplication by $i-n$ on the space $H^i(X)$.
Every ample class $\kappa \in H^2(X,\bbQ)$ defines a Lefschetz operator
$$e_\kappa :=\cup \kappa :H^\bullet (X)\to H^{\bullet +2}(X)$$
i.e. $e_\kappa ^i:H^{n-i}(X)\to H^{n+i}(X)$ is an isomorphism.
In classical Hodge theory one also considers the (unique) operator
$$f_\kappa :H^\bullet (X)\to H^{\bullet -2}(X)$$
such that $(e_\kappa ,h,f_\kappa )\subset End(H^\bullet(X))$ is an $sl_2$-triple. Let $\mathfrak{g}_{NS}(X)\subset End(H^\bullet(X))$ be the Lie algebra generated by such $sl_2$-triples $(e_\kappa ,h,f_\kappa)$ for all ample classes $\kappa \in H^2(X,\bbQ)$. This Lie algebra is graded by the adjoint action of $h$. It is called the {\it Neron-Severi} Lie algebra of $X$ \cite{LL}. This Lie algebra is semi-simple \cite[Prop.1.6]{LL}.

On the other hand, one has the group of autoequivalences of the derived category $D^b(cohX)$. This group acts naturally on the cohomology $H (X,\bbQ)$ and we denote by $G^{eq}(X)$ its image in $Gl(H (X,\bbQ))$. Let $\overline{G^{eq}(X)}\subset
Gl(H (X,\bbQ))$ be the algebraic $\bbQ$-subgroup which is the Zariski closure of $G^{eq}(X)$, and let $L^{eq}(X):=Lie\overline{G^{eq}(X)}\subset End (H(X,\bbQ))$ be its Lie algebra.

The following theorem was proved in \cite{GLO}.

\begin{thm} \label{main conj for ab var} Let $A$ be an abelian variety. Then there is an equality of Lie subalgebras of $End(H(A,\bbQ))$:
\begin{equation}\label{eq ab var}L^{eq}(A)=\mathfrak{g}_{NS}(A)
\end{equation}
\end{thm}

In \cite{Pol} the groups $\overline{G^{eq}(A)}$ were studied and classified according to the type of the abelian variety $A$. In \cite{LL} a similar classification of Lie algebras $\mathfrak{g}_{NS}(A)$ was obtained. Theorem
\ref{main conj for ab var} follows from the comparison of the two lists.

We expect a similar phenomenon for hyperkahler manifolds.

\begin{conj}\label{conj a} Let $X$ be a projective hyperkahler manifold. Then we have an equality of Lie subalgebras of $End(H^\bullet (X,\bbQ))$:
\begin{equation}\label{equality of lie subalg}
L^{eq}(X)=\mathfrak{g}_{NS}(X)
\end{equation}
\end{conj}

The conjecture is easily verified for K3 surfaces (Corollary \ref{cor conj for k3 surf}). By a recent result of Taelman \cite{Tael} it also holds for $\Hilb ^2(X)$ of a K3 surface $X$.

\begin{remark} \label{weak conj} Consider the "biggest" $\mathfrak{g}_{NS}(X)$-submodule of $H^\bullet (X,\bbQ)$, which is generated by $1\in H^0(X,\bbQ)$. Denote it by
$\mathfrak{g}_{NS}(X)\cdot 1$. A weaker version of Conjecture \ref{conj a} would say that the Lie algebra $L^{eq}(X)$ preserves this subspace and we have the equality of subalgebras of $End (\mathfrak{g}_{NS}(X)\cdot 1)$:
\begin{equation}\label{weak equality} L^{eq}(X)\vert _{\mathfrak{g}_{NS}(X)\cdot 1}=\mathfrak{g}_{NS}(X)
\end{equation}
\end{remark}

Using results of \cite{Tael} we prove (Theorem \ref{half of conj for hyperkahler}) one inclusion in this weak version of the conjecture:
\begin{equation}\label{prove half of weak form}L^{eq}(X)\vert_{\mathfrak{g}_{NS}\cdot 1}\subseteq \mathfrak{g}_{NS}(X)
\end{equation}

For a general smooth projective variety $X$ a priori it is not clear that either side of \eqref{equality of lie subalg} is contained in the other.  Conjecture \ref{conj a} is false for any positive dimensional smooth projective variety $X$ which is Fano or of  general type. Note however that for any smooth projective variety $X$ the Lie algebra of the subgroup of $\overline{G^{eq}(X)}$ corresponding to tensoring with line bundles is by definition contained in $\mathfrak{g}_{NS}(X)$ (this is the only immediately visible relation between the group $G^{eq}(X)$ and the Lie algebra $\mathfrak{g}_{NS}(X)$).

It is natural to ask if Conjecture \ref{conj a} holds for Calabi-Yau varieties. An easy counterexample is given by a smooth hypersurface $X\subset \bP ^n$ of degree $n+1$, assuming that $n=2k\geq 4$. Both $G^{eq}(X)$ and $\mathfrak{g}_{NS}(X)$ preserve the space $H^{even}(X,\bbQ)$ with its  skew-symmetric Mukai pairing. Clearly $\mathfrak{g}_{NS}(X)=sl_2$, but $L^{eq}(X)=\mathfrak{sp}(H^{even}(X,\bbQ))$ (Theorem \ref{half about geq}). So we have the strict inclusion
$$L^{eq}(X)\vert_{\mathfrak{g}_{NS}\cdot 1}\supsetneq \mathfrak{g}_{NS}(X)$$

Nevertheless we expect that Conjecture \ref{conj a} hold for "most" CY varieties as well. Let us explain why. The reason that Conjecture \ref{conj a} fails in the above counter example is that the Picard rank of $X$ is 1 and hence  $\mathfrak{g}_{NS}(X)$ is too small. The classification Theorem 6.8 in
\cite{LL} suggests that one should expect $\mathfrak{g}_{NS}(X)$ to be either small, that is to be nonzero only in degrees $-2,0,2$ (which is rare), or else to be maximal, i.e. to be the full Lie algebra preserving a nondegenerate form. For example, Conjecture \ref{conj a} should hold if in the above example one replaces $\bbP ^n$ with a smooth toric Fano variety (which is not a product) with Picard rank $\geq 2$, and take $X$ to be a smooth anticanonical divisor.

\subsection{Kontsevich's conjecture for mirror symmetric families of (weakly) CY varieties} The following sentence appears in the introduction section of \cite{BorHor}: "Kontsevich \cite{Kon} conjectured that the action on cohomology of the group of self-equivalences of the bounded derived category of coherent sheaves on a smooth projective Calabi–Yau variety matches the monodromy action on the cohomology of the mirror Calabi–Yau variety associated to the variations of complex structures."

Below we state our version of Kontsevich's conjecture (Conjecture \ref{subs main conj}). First let us  make a few definitions and reminders.

Let $\cX /S$ be a family of smooth complex projective varieties over a connected base $S$. Fixing a point $s\in S$ we get the fundamental group $\pi _1(S,s)$ and its monodromy representation
$$\mu :\pi _1(S,s)\to Gl(H^\bullet(X_s,\bbQ))$$
in the cohomology $H^\bullet(X_s,\bbQ)$ of the fiber $X_s$. We denote by $G^{mon}(\cX)$ the image of $\mu$. This is a discrete group whose isomorphism class does not depend on the choice of a point $s\in S$. It is called the monodromy group of $\cX$. Also  denote by $G^{eq}(\cX)$ the group $G^{eq}(X_s)$ for a {\it very general} fiber $X_s$ of $\cX$. (A fiber is very general if it lies outside of a countable union of analytic subvarieties of the base.)

\begin{defn} The equivalence relation $\sim $ on the collection of discrete groups is generated by allowing to replace a group $G$ by a subgroup of finite index or by the quotient of $G$ by a finite normal subgroup. If $G\sim G'$ we say that $G$ and $G'$ are {\bfseries isomorphic up to finite groups}.
\end{defn}

In the mathematical literature there exist at least two series of "mirror symmetric" (MS) families of CY varieties. Namely, one has

(I) Mirror symmetric (MS) families of lattice polarized K3 surfaces \cite{Dolg},\cite{Pink}.

(II) MS families of anticanonical divisors in dual Fano toric varieties \cite{Bat}.

Since the author is not aware of a mathematical definition of MS, we use this term in quotation marks. However, the definition of the above families is indeed based on symmetry relations: of lattices in case (I) and of polytopes in case (II). In Section \ref{sect for ab var} below we introduce a third series:

(III) MS families of abelian varieties.

This is defined using a simple symmetry relation of $\bbQ$-algebraic groups.

In our understanding the point of MS is that a simple minded duality (as in (I), (II), (III)) implies a duality relation involving highly sophisticated objects, like the derived category. Let us formulate the following principle (Kontsevich's conjecture), which we call "a conjecture" for simplicity of statement and of reference.

\begin{conj} \label{subs main conj} Let $\cX$ and $\cX ^\vee$ be mirror symmetric families of complex smooth projective varieties with trivial canonical line bundle. Then the groups $G^{mon}(\cX)$ and $G^{eq}(\cX ^\vee)$ are isomorphic up to finite groups.
\end{conj}

We prove Conjecture \ref{subs main conj} for families (I), (III) and present some evidence for it in case (II). Let us briefly summarize our results on Conjecture \ref{subs main conj} in the three examples.

\subsubsection{Lattice polarized K3 surfaces} Let $L$ be the lattice of a K3 surface. Recall \cite{Dolg} that primitive sublattices $M,M^\vee \subset L$ of signatures $(1,s)$ and $(1,18-s)$ respectively, are called {\it mirror symmetric} if
$$M^\perp _L=M^\vee \oplus U$$

Following the works \cite{Dolg},\cite{Pink} we consider the ample $M$- and $M^\vee$-polarized families, $\cU^a _M$ and $\cU^a _{M^\vee}$ respectively, of K3 surfaces. We check Conjecture \ref{subs main conj} for these families. This is just a pleasant exercise, since one knows everything about the group $G^{eq}$ and the moduli of K3 surfaces.

\subsubsection{Abelian varieties} We extend the work \cite{GLO} by defining the notion of {\it mirror symmetric families} of abelian varieties. In loc.cit. we
considered {\it algebraic pairs} $(A,\omega _A)$ where $A$ is an abelian variety and $\omega _A$ is an element of the {\it complexified ample cone} $C_A$ of $A$. Then we defined the mirror symmetry relation between algebraic pairs $(A,\omega _A)$ and $(B,\omega _B)$. One feature of this relation is the natural inclusion of algebraic $\bbQ$-groups
\begin{equation}\label{incl of alg groups}
Hdg_B\subseteq \overline{G^{eq}(A)},\quad Hdg_A\subseteq \overline{G^{eq}(B)}
\end{equation}
Now we say that the pairs $(A,\omega _A)$ and $(B,\omega _B)$ and {\it perfectly mirror symmetric} (PMS) if the inclusions \eqref{incl of alg groups} are equalities. Such PMS algebraic pairs naturally give rise to families of abelian varieties
\begin{equation}\label{dual families in intro}\cA:= \{A_{\eta _B}\ \vert \ \eta _B\in C_B\}\quad \text{and}\quad
\cB:=\{B_{\eta _A}\ \vert \ \eta _A\in C_A\}
\end{equation}
with bases $C_B$ and $C_A$ respectively. We call these families {\it mirror symmetric} (Definition \ref{defn of mirror sym for ab var}). A proof of Conjecture \ref{subs main conj} for such families is in Section \ref{proof of conj for fam of ab var}.

\subsubsection{Anticanonical hypersurfaces in Fano toric varieties} Batyrev \cite{Bat} constructs mirror symmetric families of CY varieties in the following way: he starts with two dual lattices $M\simeq \bbZ ^{n+1}$ and $N=M^*$ and a pair of dual reflexive polytopes $\Delta \subset M_\bbQ$, $\Delta ^\vee \subset N_\bbQ$. These dual polytopes define a pair of projective Gorenstein toric Fano varieties $\bbP _\Delta$ and $\bbP _{\Delta ^\vee}$. The induced families $\cX$ and $\cX ^\vee$ of anticanonical divisors consist of Gorenstein CY varieties. These are the MS families (see Section \ref{CY} for details).

Assume that the toric variety $\bbP _\Delta$ is smooth. Then the family $\cX$ will consist of smooth CY varieties $X$ and the group $G^{eq}(\cX)$ is defined. Assume in addition that $n$ is odd. Then for any member $X$ of the family $\cX$ we have
$G^{eq}(X)\subset Sp(H^{even}(X,\bbQ))$ and also $G^{mon}(\cX ^\vee)\subset Sp(H^n(X^\vee ,\bbQ))$. The spaces $H^{even}(X,\bbQ)$ and $H^n(X^\vee ,\bbQ)$ are isomorphic and we expect that $G^{eq}(\cX)$ and $G^{mon}(\cX ^\vee)$ are arithmetic subgroups in the corresponding isomorphic symplectic groups.
 We prove a little weaker statement for the group $G^{eq}(X)$. Namely we show that the Zariski closure $\overline{G^{eq}(X)}$ is equal to $Sp(H^{even}(X,\bbQ))$ for any $X$ in the family $\cX$. Some evidence is also provided for the group $G^{mon}(\cX ^\vee)$.

\subsubsection{}
Strictly speaking, among the 3 families mentioned above, Conjecture \ref{subs main conj} can be tested only in case (II), where the actual universal family exists. In cases (I) and (III) one typically has only the coarse moduli space $\overline{S}$, which is the quotient by a discrete group of an analytic space $S$ that is a base of an actual family $\cX \to S$. (For example, $\overline{S}$ can the quotient of the Lobachevsky upper half plane $S$ by the group $SL(2,\bbZ)$ and $\cX \to S$ the natural family of elliptic curves).
So in order to make Conjecture \ref{subs main conj} applicable to cases (I) and (III) we need to extend appropriately the notion of the monodromy. This is done in \ref{general monodromy} below.

\subsubsection{} Some aspects of Kontsevich's conjecture were already studied in \cite{Hor} and in \cite{Szen}. However, the intersection of results in loc. cit. with ours appears to be minimal.

\subsection{Organization of the paper}
Section 2 collects some general results on Fourier-Mukai transforms and gives the definition of the monodromy in a somewhat nonstandard situation. In section 3 we discuss Conjecture \ref{conj a} for hyperkahler varieties. Sections 4 and 5 deal with mirror symmetry for families of K3 surfaces and abelian varieties respectively. In both cases the Conjecture \ref{subs main conj} is proved. In the last Section 6 we discuss the case of Calabi-Yau hypersurfaces in toric varieties and prove some partial results in the direction of the conjecture.

\subsubsection{Acknowledgments} Alex Furman helped us with the proof of Lemma \ref{furman}. Michael Larsen, Igor Dolgachev, Alexander Efimov, Daniel Huybrechts and Victor Batyrev answered many of our questions. Paul Horja sent us his notes of Kontsevich's talk \cite{Kon} and Balazs Szendroi informed us of the paper \cite{Szen}. It is our pleasure to thank all of them. We are also grateful to the anonymous referee for careful reading of the manuscript, finding mistakes, and making useful remarks and suggestions.

\section{Some generalities and extension of the notion of monodromy}

\subsection{Notation} \label{notation} We consider smooth complex projective varieties. For such a variety $X$, $H (X,\bbQ)$ denotes the singular cohomology of the corresponding analytic manifold. The bounded derived category of coherent sheaves on $X$ is denoted by $D^b(X)$. Its group of autoequivalences is $AutD^b(X)$. There exists the Chern character map
$$ch:D^b(X)\to H(X,\bbQ)$$
Let $\sqrt{td_X}\in H(X,\bbQ)$ be the square root of the Todd class of $X$. For $F\in D^b(X)$ one defines its {\it Mukai vector}  $v(F)$ \cite[5.28]{Huyb} as
$$v(F):=ch(F)\cup \sqrt{td_X}\in H(X,\bbQ)$$

\subsection{Action of the group $AutD^b(X)$ on the cohomology $H(X,\bbQ)$}

There is a natural homomorphism of groups $\rho _X:AutD^b(X)\to Gl(H(X,\bbQ))$. Let us recall it.

Consider the two projections $X\stackrel{p}{\leftarrow} X\times X\stackrel{q}{\rightarrow} X$. It is known \cite{Or1} that any autoequivalence $\Phi \in AutD^b(X)$ is given by a Fourier-Mukai functor $\Phi _E$ for a unique kernel $E\in D^b(X\times X)$. That is $$\Phi (-)=\Phi _E(-):={\bf R}q_*(p^*(-)\stackrel{\bf L}{\otimes} E)$$
This operation is compatible with the Mukai vector in the following sense. Any $e\in H(X\times X)$ defines the corresponding cohomological transform $\Phi ^H_e$
$$\Phi ^H_e(-)=q_*(p^*(-)\cup e):H(X,\bbQ)\to H(X,\bbQ)$$
Then for any kernel $E\in D^b(X\times X)$, and $F\in D^b(X)$ we have
$$\Phi ^H_{v(E)}(v(F))=v(\Phi _E(F))$$
\cite[5.29]{Huyb}
and the correspondence $\Phi _E\mapsto \Phi _{v(E)}$ is the group homomorphism
$$\rho _X:AutD^b(X)\to Gl(H(X,\bbQ))$$
\cite[5.32]{Huyb}. We note that the action of $AutD^b(X)$ on $H(X,\bbQ)$ preserves the {\it Mukai pairing} \cite[5.44]{Huyb}, which is nondegenerate and is a modification of the Poincare pairing \cite[5.42]{Huyb}.

\begin{defn} \label{defn of geq} Denote by $G^{eq}(X)$ the image of the homomorphism $\rho _X$.
\end{defn}

The group $G^{eq}(X)$ rarely preserves the integral cohomology $H(X,\bbZ)\subset H(X,\bbQ)$ but it preserves a different lattice.
Consider the topological K-group $K_{top}(X)=K^0_{top}(X)\oplus K^1_{top}(X)$ and the map
$$v_{top}:K_{top}(X)\to H(X,\bbQ),\quad F\mapsto \sqrt{td_X}\cdot ch(F)$$
The image $im(v_{top})$ is a lattice in $H(X,\bbQ)$ of full rank. This lattice is preserved by the group $G^{eq}(X)$ \cite{AdTho}.
In particular, $G^{eq}(X)$ is a {\it discrete} subgroup of $GL(H(X,\bbQ))$.


\subsection{Monodromy group} \label{general monodromy} Besides considering mirror dual universal families of CY varieties we also want to study the case when only coarse moduli spaces exist. Let us make a rather ad hoc definition of the monodromy group in that case. The definition seems reasonable and suffices for our purposes.

Let $f:\cX\to S$ be a continuous map of topological spaces which is a locally trivial fibration and whose fibers are compact complex manifolds (with certain additional structure, for example, an embedding in a projective space (a polarization) or a multi-polarization, or a fixed sublattice in the Neron-Severi group). Assume that $S$ is connected. Assume also that the (graded) local system ${\bf R}^\bullet f_* \bbQ _{\cX}$ is trivial. Let $G$ be a  discrete  group that acts on $S$ and this action lifts to an action on the local system ${\bf R}^\bullet f_* \bbQ _{\cX}$.
 Let $K\subset G$ denote the kernel of the $G$-action on $S$.
So elements of $K$ act by fiberwise automorphisms of the local system ${\bf R}^\bullet f_* \bbQ _{\cX}$.
Suppose that the following holds:
\begin{itemize}
\item[1.] The $G$-action on the space of global sections $H^0(S,{\bf R}^\bullet f_* \bbQ _{\cX})$ is effective (i.e. every $1\neq g\in G$ acts nontrivially).
\item[2.] The $G/K$-action on $S$ is generically free. More precisely there exists a countable union $Z\subset S$ of closed subsets such that the complement  $S^0:=S\backslash Z$ is everywhere dense and  $G/K$-action on $S^0:=S\backslash Z$ is free.
\item[3.] The quotient space $\overline{S^0}:=S^0/G=S^0/(G/K)$ is the coarse moduli space of  complex manifolds (with the given additional structure) appearing as fibers in the family $f$ over $S^0$ (that is, points of $\overline{S^0}$ are in bijection with isomorphism classes of fibers in the family $\cX\vert_{S^0}$.)
\end{itemize}

\begin{defn} \label{def mon} In the above situation we call $G/K$  the {\bfseries monodromy group} of the family $f:\cX\to S$. We denote this group $G^{mon}(\cX)$. (In case the $G/K$-action on $S^0$ is free only modulo a finite kernel, we say the $G/K$ is the {\bfseries monodromy group up to finite groups}).
\end{defn}

If the $G$-action of $S$ is free and $\overline{S}=S/G$ is a fine moduli space, i.e. the family $f:\cX\to S$ descends to a universal family $\overline{f}:\overline{\cX}\to \overline{S}$, the group $G=G^{mon}(\cX)$ coincides with the monodromy group of the local system ${\bf R}^\bullet \overline{f}_* \bbQ _{\overline{\cX}}$.

For a nontrivial example one can take $f:\cX \to S$ to be the natural family of elliptic curves over the Lobachevsky upper half plane $S$. Then the discrete group $G=SL(2,\bbZ)$ acts on $S$ and the quotient $\overline{S}=S/G$ is the coarse moduli space of elliptic curves. This $G$ action is not free, so there is no universal family $\overline{f}:\overline{\cX}\to \overline{S}$ and $G$ is not the topological fundamental group of $\overline{S}$. However,
according to Definition \ref{def mon}, $G$ is the monodromy group of the family $f:\cX\to S$. This fits well with Conjecture \ref{subs main conj}. Indeed, the family of elliptic curves  $f:\cX\to S$ is mirror symmetric to itself (Definition \ref{defn of mirror sym for ab var}) and the group $G^{eq}(E)$ of a general elliptic curve is the group $G$.

\begin{remark} We will show that in case of mirror symmetric families of abelian varieties or lattice polarized families of K3 surfaces there exists the monodromy group in the sense of Definition \ref{def mon}.
\end{remark}

\section{Conjecture \ref{conj a} for hyperkahler manifolds}

\subsection{Construction of the Lie algebras $\mathfrak{g}_{NS}$ and $\mathfrak{g}_{tot}$ for a hyperkahler manifold}\label{prelim on hyperk}

For a smooth projective variety $X$ one defines the {\it total} Lie algebra (sometimes called the LLV Lie algebra) $\mathfrak{g}_{tot}(X)\subset End(H(X,\bbQ))$ in the same way as $\mathfrak{g}_{NS}(X)$ but using all Lefschetz elements in $H^2(X,\bbQ)$ and not just in $NS(X)$ \cite{LL},\cite{Verb}. We recall the description of these Lie algebras for a hyperkahler manifold.

Let $V$ be a finite dimensional $\bbQ$-vector space with a nondegenerate symmetric bilinear form $q$. Consider the graded vector space
$$\tilde{V}:=\bbQ  e\oplus V\oplus \bbQ  \eta$$
where $\deg(e)=0,$ $\deg(V)=2,$ $\deg(\eta) =4$. Extend the form $q$ to a form $\tilde{q}$ on $\tilde{V}$ by putting $\tilde{q}(e,\eta)=1$, $\tilde{q}(e,V)=\tilde{q}(\eta ,V)=0$.

We make $\tilde{V}$ into a graded commutative algebra by defining multiplication
$$x e:=x,\quad \eta e:=\eta ,\quad x y:=q(x,y) \eta$$
for $x,y\in V$. Every nonisotropic $x\in V$ defines a Lefschetz operator on $\tilde{V}$, hence gives rise to an $sl_2$-triple. All such triples generate a graded Lie subalgebra $\mathfrak{g}(V)\subset End(\tilde{V})$. This is a graded Lie algebra
$$\mathfrak{g}(V)=\mathfrak{g}(V)_{-2}\oplus \mathfrak{g}(V)_{0}\oplus \mathfrak{g}(V)_2$$
and $\mathfrak{g}(V)=\mathfrak{so}(\tilde{V},\tilde{q})$ \cite[Sect.9]{Verb}. Moreover, $\mathfrak{g}(V)_{0}=\mathfrak{so}(V,q)\oplus \bbQ$.

If $V'\subset V$ is a subspace such that the form $q':=q\vert _{V'}$ is nondegenerate, consider a similar extension $(\tilde{V'},\tilde{q'})\subset (\tilde{V},\tilde{q})$. One can generate a Lie subalgebra $\mathfrak{g}(V')\subset End(\tilde{V})$ by using only the Lefschetz operators from $V'$. Then again $\mathfrak{g}(V')=\mathfrak{so}(\tilde{V'},\tilde{q'})$ and $\mathfrak{g}(V')_{0}=\mathfrak{so}(V',q')\oplus \bbQ$.

The above construction is applicable to any smooth projective surface $Y$. Namely, by taking $V=H^2(Y,\bbQ)$ and $V'=NS(Y)_\bbQ$ we get
$$\mathfrak{g}(V)=\mathfrak{g}_{tot}(Y),\quad \mathfrak{g}(V')=\mathfrak{g}_{NS}(Y)$$

More interestingly, if $X$ is a projective hyperkahler manifold, $V=H^2(X,\bbQ)$ with the Bogomolov-Beauville (BB) form $q_X$ and $V'=NS(X)_\bbQ$, we again obtain \cite{Verb}, \cite{LL}:
$$\mathfrak{g}(V)=\mathfrak{g}_{tot}(X),\quad \mathfrak{g}(V')=\mathfrak{g}_{NS}(X)$$
In this case the extended lattice $(\tilde{H}^2(X,\bbQ),\tilde{q}_X)$ is called the rational Mukai lattice of $X$. It has the obvious integral structure
$$\Lambda =\bbZ e\oplus H^2(X,\bbZ)\oplus \bbZ \eta \subset \tilde{H}^2(X,\bbQ)$$
and we equip it with the Hodge structure of weight zero:
\begin{equation}\label{def of hdg str on tilde}\tilde{H}^2(X,\bbQ)=\bbQ e\oplus H^2(X,\bbQ (1))\oplus \bbQ \eta
\end{equation}
Denote by $O_{hdg}(\Lambda)$ the discrete group of Hodge isometries of $\Lambda$ and let  $\overline{O_{hdg}(\Lambda)}\subset O(\tilde{H}^2(X,\bbQ),\tilde{q}_X)$ be its Zariski closure.

\begin{lemma} \label{easy useful lemma} Let $X$ be a projective hyperkahler manifold.  Then we have the
equality of Lie subalgebras of $End(\tilde{H}^2(X,\bbQ))$
\begin{equation}\label{contains}
Lie(\overline{O_{hdg}(\Lambda)})= \mathfrak{g}_{NS}(X)
\end{equation}
\end{lemma}

\begin{proof} Recall that the signature of the BB form $q_X$ on $H^2(X,\bbZ)$ is $(3,b_2-3)$ and the signature of its restriction to $NS(X)$ is $(1,s)$ \cite{Huyb2}. Denote by $T(X)\subset H^2(X,\bbZ)$ the orthogonal complement of $NS(X)$ in $H^2(X,\bbZ)$. Then
$$H^2(X,\bbQ)=NS(X)_\bbQ \oplus T(X)_\bbQ$$
and the signature of the restriction of $q_X$ to $T(X)$ is $(2,b_2-3-s)$.
Moreover the Hodge structure on $H^2(X,\bbZ)$ restricts to one on $T(X)$.
Let $O_{hdg}(T(X))$ be the corresponding group of Hodge isometries.

First we claim that the group $O_{hdg}(T(X))$ is finite. We copy the argument from \cite[3.3.4]{Huyb3}: Consider the real space $T(X)_\bbR$ and its orthogonal decomposition $T(X)_\bbR =W\oplus W^\perp$
where $W=H^2(X,\bbR)\cap (H^{2,0}\oplus H^{0,2})$. The form $q$ is positive definite on $W$ and hence is negative definite on $W^\perp$.
The group $O_{hdg}(T(X))$ preserves $W$ and $W^\perp$ so it is contained in a finite subgroup of $O(W)\times O(W^\perp)$.

Define the sublattice
$$\Lambda ':=NS(X)\oplus \bbZ e\oplus \bbZ \eta \subset \Lambda$$
Then the sublattices $\Lambda '$ and $T(X)$ are preserved by the group $O_{hdg}(\Lambda)$ and $\Lambda '\oplus T(X)\subset \Lambda$ is a sublattice of full rank. Consider the restriction homomorphism
$$r:O_{hdg}(\Lambda)\to O(\Lambda ')$$
As explained above the kernel of $r$ is finite. It is also clear that the image of $r$ is a subgroup of finite index in $O(\Lambda ')$.  Therefore
$$Lie(\overline{O_{hdg}(\Lambda)})=Lie (\overline{O(\Lambda ')})=
\mathfrak{so}(\Lambda '_\bbQ)= \mathfrak{g}_{NS}(X)$$
\end{proof}

\begin{remark} In Lemma \ref{easy useful lemma} the hyperkahler manifold $X$ was used only through the associated Hodge structure on $H^2(X)$. So essentially it is a statement about K3-type Hodge structures.
\end{remark}

Later we will need the following fact.

\begin{lemma} \label{for future use} Let $\Gamma \subset O(\tilde{H}^2(X,\bbQ),\tilde{q}_X)$ be a {\bf discrete} subgroup of Hodge isometries, and  let $\overline{\Gamma }\subset O(\tilde{H}^2(X,\bbQ),\tilde{q}_X)$ be its Zariski closure. Then
$$Lie \overline{\Gamma }\subset \mathfrak{g}_{NS}(X)$$
\end{lemma}

\begin{proof} Similar to the proof of Lemma \ref{easy useful lemma}. Namely, in the above notation consider the orthogonal decomposition of rational Hodge structures
\begin{equation} \label{decomp}\tilde{H}^2(X,\bbQ)=\Lambda '_\bbQ \oplus
T(X)_\bbQ
\end{equation}
This decomposition is preserved by the group $O_{hdg}(\tilde{H}^2(X,\bbQ),\tilde{q}_X)$, and so
$$O_{hdg}(\tilde{H}^2(X,\bbQ),\tilde{q}_X)=O(\Lambda '_\bbQ)\times O_{hdg}(T(X)_\bbQ)$$
As in the proof of Lemma \ref{easy useful lemma} we conclude that the group   $\Gamma \cap O_{hdg}(T(X)_\bbQ)$ is finite. Therefore
$$Lie \overline{\Gamma }\subset Lie (\overline{O(\Lambda ')})=
\mathfrak{so}(\Lambda '_\bbQ)= \mathfrak{g}_{NS}(X)$$
\end{proof}

\begin{cor}\label{cor conj for k3 surf} Conjecture \ref{conj a} holds for projective K3 surfaces.
\end{cor}

\begin{proof} Let $X$ be a projective K3 surface. Then $G^{eq}$ is a discrete subgroup of $GL(\tilde{H}^2(X,\bbZ),\tilde{q}_X)$ \cite[Ch.10]{Huyb}. Moreover $G^{eq}$ is a subgroup of $O_{hdg}(\tilde{H}^2(X,\bbZ))$ and its index is at most 2 \cite[Ch.10]{Huyb}. It remains to apply Lemma \ref{easy useful lemma}.
\end{proof}

\subsection{Towards a proof of Conjecture \ref{conj a} for hyperkahler manifolds}

The following theorem establishes one inclusion in the weak form of Conjecture \ref{conj a} (see Remark \ref{weak conj}).

\begin{theorem} \label{half of conj for hyperkahler} Let $X$ be a hyperkahler manifold of dimension $2d$. The action of the Lie algebra $L^{eq}(X)$ on $H(X,\bbQ)$ preserves the subspace $\mathfrak{g}_{NS}(X)\cdot 1$. Moreover we have the inclusion of rational Lie subalgebras of $End(\mathfrak{g}_{NS}(X)\cdot 1)$:
\begin{equation}
\label{inclusion of lie}L^{eq}(X)\vert _{\mathfrak{g}_{NS}(X)\cdot 1}\subset \mathfrak{g}_{NS}(X).
\end{equation}
\end{theorem}

\begin{proof} We will use some results of \cite{Tael}.

\begin{theorem} \cite[Thm.A,B]{Tael}\label{thm from tael} Let $\Phi :D(X)\to D(X)$ be an autoeqvivalence and $\Phi ^H\in GL(H(X,\bbQ)$ the corresponding operator on the cohomology. Then the following holds.

(1) The operator $Ad_{\Phi ^H} $ preserves the Lie subalgebra $\mathfrak{g}_{tot}(X)\subset End (H(X,\bbQ))$.

(2) $\Phi ^H$ preserves the irreducible $\mathfrak{g}_{tot}(X)$-submodule $\mathfrak{g}_{tot}(X)\cdot 1\subset H(X,\bbQ)$.
\end{theorem}

It follows from part (2) of Theorem \ref{thm from tael} that there is
 a group homomorphism $G^{eq}(X)\to GL(\mathfrak{g}_{tot}(X)\cdot 1)$. Denote by $G^{eq}(X)\vert_{\mathfrak{g}_{tot}(X)\cdot 1}$ its image.

The subspace $\mathfrak{g}_{tot}(X)\cdot 1\subset H^{even} (X,\bbQ)$ inherits the Hodge structure of weight zero, given by
$$H^{even} (X,\bbQ)=\bigoplus _sH^{2s}(X,\bbQ (s))$$
The group $G^{eq}(X)\vert_{\mathfrak{g}_{tot}(X)\cdot 1}$ is a discrete group of Hodge isometries of $\mathfrak{g}_{tot}(X)\cdot 1$.

Theorem \ref{thm from tael} implies that the conjugation action of the group
$G^{eq}(X)\vert_{\mathfrak{g}_{tot}(X)\cdot 1}$ gives the group homomorphism
$$\alpha :G^{eq}(X)\vert_{\mathfrak{g}_{tot}(X)\cdot 1}\to Aut(\mathfrak{g}_{tot}(X))$$
and an element $g\in \ker(\alpha)$ is an automorphism of the simple $\mathfrak{g}_{tot}(X)$-module $\mathfrak{g}_{tot}(X)\cdot 1$. Hence $g$ is a scalar operator on $\mathfrak{g}_{tot}(X)\cdot 1$. But $g$  is an isometry, so $g=\pm 1$.

The Lie algebra $\mathfrak{g}_{tot}(X)$ is simple, so $G^{eq}(X)\vert_{\mathfrak{g}_{tot}(X)\cdot 1}$ has a subgroup $P$ of finite index whose image under $\alpha $ is contained in the
 adjoint group $Ad(\mathfrak{g}_{tot}(X))$ of the Lie algebra $\mathfrak{g}_{tot}(X)$.

Let $G_{tot}(X)\subset Gl(\mathfrak{g}_{tot}(X)\cdot 1)$ be the connected Lie subgroup with the Lie algebra $\mathfrak{g}_{tot}(X)$. The adjoint surjective  homomorphism $\beta :G_{tot}(X)\to Ad(\mathfrak{g}_{tot}(X))$ has a finite kernel. Put $A:=\beta ^{-1}(\alpha (P))\subset G_{tot}(X)$. So we have the diagram
$$\begin{array}{rcl} P & \stackrel{\alpha}{\to} & Ad(\mathfrak{g}_{tot}(X))\\
 & & \uparrow \beta \\
A= \beta ^{-1}(\alpha (P)) & \hookrightarrow & G_{tot}(X)
\end{array}
$$
The group $P$ is a discrete group of Hodge isometries of $\mathfrak{g}_{tot}\cdot 1$.
If $p\in P,a\in A$ are such that $\alpha (p)=\beta (a)$, then $pa^{-1}$ acts as a scalar on $\mathfrak{g}_{tot}\cdot 1$. It follows that $A$ is a discrete subgroup of $G_{tot}(X)$ which acts by Hodge isometries on $\mathfrak{g}_{tot}\cdot 1$ and we have the equality of Lie subalgebras of $End (\mathfrak{g}_{tot}\cdot 1)$:
$$Lie\overline{A}=L^{eq}(X)\vert_{\mathfrak{g}_{tot}(X)\cdot 1}$$
For the proof of the theorem it suffices to establish the inclusion of Lie subalgebras of $End(\mathfrak{g}_{tot}\cdot 1)$:
\begin{equation}\label{inclusion of lie alg} Lie\overline{A}\subset \mathfrak{g}_{NS}(X)
\end{equation}

Recall a lemma from \cite{Tael}.

\begin{lemma} \label{lemma on connection} Let $\dim X=2d$. Then there exists a unique map
$$\Psi :\mathfrak{g}_{tot}\cdot 1\to \Sym ^d\tilde{H}^2(X,\bbQ)$$
with the following properties.

(1) $\Psi (1)=e^d/d!$

(2) $\Psi $ is a morphism of $\mathfrak{g}_{tot}$-modules.

This map is an injective isometry and a morphism of Hodge structures \eqref{def of hdg str on tilde}.
 \end{lemma}

\begin{proof} See \cite[Prop. 3.5, 3.7, Lemma 4.6]{Tael}.
\end{proof}

Denote by $\tilde{G}_{tot}(X)\subset GL(\tilde{H}^2(X,\bbQ))$
the connected algebraic subgroup with the Lie algebra $\mathfrak{g}_{tot}(X)\subset End(\tilde{H}^2(X,\bbQ))$. The group $\tilde{G}_{tot}(X)$ acts naturally on the space $\Sym ^d\tilde{H}^2(X,\bbQ)$, and the restriction to the subspace $\Psi (\mathfrak{g}_{tot}\cdot 1)$ gives (by Lemma \ref{lemma on connection}) a surjective group homomorphism with finite kernel
$$\theta :\tilde{G}_{tot}(X)\to G_{tot}(X)$$
Put $B=\theta ^{-1}(A)\subset \tilde{G}_{tot}(X)$. This is a discrete sugbroup of isometries of $\tilde{H}^2(X,\bbQ)$. We claim that it also preserves the Hodge structure. Indeed, the Hodge structure on $\tilde{H}^2(X,\bbQ)$ is given by a group homomorphism $h:S^1\to SO(H^2(X,\bbR))\subset \tilde{G}_{tot}(X)(\bbR)$. The induced  Hodge structure on the subspace $\Psi (\mathfrak{g}_{tot}(X)\cdot 1)\subset
\Sym ^d\tilde{H}^2(X,\bbQ)$ is given as the
composition
$$S^1\stackrel{h}{\to } \tilde{G}_{tot}(X)(\bbR)\stackrel{\theta}{\to} G_{tot}(X)(\bbR)\subset
Gl(\Psi(\mathfrak{g}_{tot}\cdot 1),\bbR)$$
For every $b\in B$, the element $\theta (b)\in A$ commutes with the Hodge structure on $\Psi(\mathfrak{g}_{tot}\cdot 1)$ (since $\Psi$ is a morphism of Hodge structures). It follows that $b$ commutes with the Hodge structure on $\tilde{H}^2(X,\bbQ)$ up to an element in the kernel of $\theta$, which is a finite group. But the group $S^1$ is connected, hence $b$ and the image of $S^1$ commute.

We conclude that $B$ is a discrete group of Hodge isometries of $\tilde{H}^2(X,\bbQ)$.
Lemma \ref{for future use} implies that
$$Lie\overline{B}\subset \mathfrak{g}_{NS}(X)$$
But then
$$Lie\overline{A}=Lie\overline{B}\subset \mathfrak{g}_{NS}(X)$$
which proves Theorem \ref{half of conj for hyperkahler}.
\end{proof}

\section{Conjecture \ref{subs main conj} for dual families of K3 surfaces}
We define the notion of mirror symmetric families of lattice polarized families of K3 surfaces following the work of Dolgachev-Nikulin \cite{Dolg} and Pinkham \cite{Pink}. Then we prove Conjecture \ref{subs main conj} for such families (Theorem \ref{main thm for k3}).

Let us recall the notion of a lattice polarized K3 surface and their moduli spaces following \cite{Dolg}. First we review the classical theory of moduli space of K3 surfaces \cite{K3}. Let $L$ be the even unimodular lattice of signature $(3,19)$ which is the direct sum
$$L=(-E_8)^{\oplus 2}\oplus (U)^{\oplus 3}$$
Recall that for any K3 surface $X$ the lattice $H^2(X,\bbZ)$ is isomorphic to $L$. Unless stated otherwise we consider K3 surfaces which are not necessarily algebraic.

\begin{defn} A {\bfseries marked} K3 surface $(X,u)$ is a K3 surface $X$ with an isomorphism of lattices $u:H^2(X,\bbZ)\to L$. Marked surfaces $(X,u)$ and $(X',u')$ are isomorphic if there exists an isomorphism $f:X\to X'$ such that
$u'=u\cdot f^*$.
\end{defn}

The following theorem is proved in \cite[Exp. XIII]{K3}.

\begin{thm} \label{general moduli space} There exists a fine moduli space $\mathfrak{M}$ of marked K3 surfaces.
\end{thm}

The moduli space $\mathfrak{M}$ is a non-separated analytic space. By definition it comes with the universal family $f:\cU \to \mathfrak{M}$ of marked K3 surfaces.
The orthogonal group $\Gamma =O(L)$ acts naturally of $\mathfrak{M}$ by changing the marking $\gamma \cdot (X,u)=(X,\gamma \cdot u)$ and the quotient $\mathfrak{M}/\Gamma$ is the set of isomorphism classes of K3 surfaces, i.e. $\mathfrak{M}/\Gamma$ is the coarse moduli space of K3 surfaces. However, the action of $\Gamma $ on $\mathfrak{M}$ is not proper (because the stabilizer of a point $(X,u)$ is isomorphic to the automorphism group of $X$, which may be infinite) and there is no reasonable analytic structure on the set $\mathfrak{M}/\Gamma$.

The space $\mathfrak{M}$ has two connected components which are interchanges by the involution $(X,u)\mapsto (X,-u)$. Choose one of these components $\mathfrak{M}^0$ and let $\Gamma ^0\subset \Gamma $ be its stabilizer (a subgroup of index 2). Clearly $\mathfrak{M}^0/\Gamma ^0=\mathfrak{M}/\Gamma $. We denote by $f^0:\cU ^0\to \mathfrak{M}^0$ the restriction of the universal family $f:\cU \to \mathfrak{M}$.

Given a marked K3 surface $(X,u)$, the image of the line $H^{2,0}(X)$ under the map $u_\bbC :H^2(X,\bbC) \to L_\bbC =\bbC ^{22}$ defines a point in the corresponding projective space $\bbP (L_\bbC)=\bbP ^{21}$. This point lies in the {\it period domain} $\Omega \subset \bbP ^{21}$ consisting of points $\{\omega \in \bbP ^{21}\ \vert \ (\omega ,\omega )=0,\ (\omega ,\overline{\omega})>0\}$ and so one gets the {\it period map}
\begin{equation}\label{period map}
P:\mathfrak{M}\to \Omega
\end{equation}
It is known that the map $P$ is holomorphic, etale and surjective
\cite[Ex. XIII]{K3}. Its restriction to $\mathfrak{M}^0$ is also surjective.

\begin{lemma} \label{eigenspace for g} Let $g\in O(L)$, $g\neq \pm1$. The collection of marked K3 surfaces
$$E_g:=\{(X,u) \vert \ \text{$u_\bbC (H^{2,0})$ is an eigenvector of $g_\bbC$}\}$$
is contained in a proper analytic subspace of $\mathfrak{M}$.
\end{lemma}

\begin{proof} It suffices to prove that the image $P(E_g)$ is contained in a proper analytic subspace of the period domain $\Omega$. Our assumption on $g$ means that it is not a scalar operator. Thus eigenvectors of $g_\bbC$ are contained in a union of proper linear subspaces of $L_\bbC$. But the period domain $\Omega$, being an open subset of a nondegenerate quadric, is not contained in any hyperplane in $\bbP ^{21}$, which proves the lemma.
\end{proof}

\begin{cor} \label{monodr for general k3} Consider $f^0:\cU ^0\to \mathfrak{M}^0$ as the family of {\bfseries unmarked} K3 surfaces. Then the group $\Gamma ^0$ is its  monodromy group (Definition \ref{def mon}).
\end{cor}

\begin{proof} In terms of  Definition \ref{def mon} we have $S=\mathfrak{M}^0$, $\cX =\cU ^0$, $G=\Gamma ^0$, $K=\{1\}$. The marking defines a canonical trivialization of the local system ${\bf R}^\bullet f^0_*\bbQ _{\cU ^0}$. Clearly the $\Gamma ^0$-action on $\bbQ \oplus L_\bbQ \oplus \bbQ=H^0(\mathfrak{M},\bbR ^\bullet f^0_*\bbQ _{\cU ^0})$ is effective. Since $\mathfrak{M}^0/\Gamma ^0$ is a coarse moduli space of K3 surfaces, it remains to show that away from a countable number of analytic subsets the $\Gamma ^0$-action on $\mathfrak{M}^0$ is free.

Let $1\neq g \in \Gamma ^0$ and assume that $(X,u)\in (\mathfrak{M}^0)^g$. For simplicity of notation let us identify $H^2(X,\bbZ)$ with $L$ by means of $u$. Then there exists an automorphism $\phi :X\to X$ such that $\phi ^* :H^2(X,\bbC)\to H^2(X,\bbC)$ equals $g_\bbC $.
In particular the line $H^{2,0}(X)$ is contained in an eigenspace of $\phi ^*$. Lemma \ref{eigenspace for g} implies that $(X,u)$ belongs to a proper analytic subspace $E_g$ of $\mathfrak{M}^0$, unless $g=\pm 1$.
We excluded the case $g=1$, and $g=-1$ does not belong to $\Gamma ^0$.

Since $\Gamma ^0$ has countably many elements, the subset of $\mathfrak{M}^0$ on which the $\Gamma ^0$-action is free is the complement of countably many proper analytic subsets, hence in particular it is everywhere dense.
\end{proof}

\subsection{Lattice polarized K3 surfaces and their moduli spaces}
Let $X$ be a projective K3 surface. It is known that the first Chern class map
$$c:Pic(X)\to H^2(X,\bbZ)$$
is injective. By Hodge index theorem $Pic(X)$ is a lattice of signature $(1,t)$.

Let $M$ be an even non-degenerate lattice of signature $(1,s)$. Let
$$\Delta (M)=\{\delta \in M\mid (\delta ,\delta)=-2\}$$
Fix a subset $\Delta (M)^+\subset \Delta (M)$ such that

(i) $\Delta (M)=\Delta ^+\coprod (-\Delta (M)^+)$;

(ii) if $\delta _1,...,\delta _k\in \Delta (M)^+$ and $\delta =\sum n_i\delta _i\in \Delta (M)$ with $n_i\geq 0$, then $\delta \in \Delta (M)^+$.

The choice of a subset $\Delta (M)^+\subset \Delta (M)$ defines the subset
$$C(M)^+=\{h\in M\mid (h,h)>0 \ \text{and $(h,\delta )>0$ for all $\delta \in \Delta (M)^+$}\}$$

\begin{defn} An $M$-{\bfseries polarized} K3 surface is a pair $(X,j)$, where $X$ is K3 surface and $j:M\hookrightarrow Pic(X)$ is a primitive lattice embedding. We say that $(X,j)$ is {\bfseries ample polarized} if in addition $j(M)$ contains the class of an ample divisor on $X$.
Two $M$-polarized K3 (resp. ample polarized) surfaces $(X,j)$ and $(X',j')$ are called isomorphic if there exists an isomorphism $f:X\to X'$ such that $j=f^*\cdot j'$.
\end{defn}

\begin{remark} \label{remark polarized means projective} Notice that any $M$-polarized K3 surface $X$ is projective. Indeed, by the signature assumption there exists $q\in M$ such that $(q,q)>0$. So there exists a line bundle $\cL \in Pic (X)$ with $c_1(\cL)^2>0$. This implies that $X$ is projective \cite[Thm. 8]{Kod}.
\end{remark}

Now assume that we are {\bf given a primitive embedding} of lattices $a:M\hookrightarrow L$.

\begin{defn} A {\bfseries marked $M$-polarized} K3 surface is a triple $(X,j,u)$ such that $(X,u)$ is a marked K3 surface, $(X,j)$
is an $M$-polarized K3 surface and in addition
$$a=u\cdot j:M\to L$$
We say that $(X,j,u)$ is {\bfseries marked ample $M$-polarized} if $(X,j)$ is ample $M$-polarized.
Two marked $M$-polarized K3 surfaces are isomorphic if they are isomorphic as marked K3 surfaces (and hence also as $M$-polarized K3 surfaces).
\end{defn}

Clearly, a marked $M$-polarized K3 surface surface $(X,j,u)$ is uniquely determined by the corresponding marked K3 surface $(X,u)$.

Let $N:=M^\perp _L$ be the orthogonal complement of $M$ in $L$. We have the inclusion of projective spaces $\bbP (N_\bbC )\subset \bbP (L_\bbC)$ and put $\Omega _M:=\Omega \cap \bbP (N_\bbC)$. This is the {\it period domain} for $M$-polarized K3 surfaces. It has 2 connected components.

For any $\delta \in \Delta (N):=\{a\in N\mid (a,a)=-2\}$ set
$$H_\delta :=\{z\in N_{\bbC}\mid (z,\delta )=0\}$$
and define
$$\Omega ^0_M:=\Omega _M\backslash \left(\bigcup _{\delta \in \Delta (N)}H_\delta \cap \Omega _M\right)$$
Since $\Omega _M$ has two connected components, so does $\Omega ^0_M$.

Similarly to Theorem \ref{general moduli space} one can prove the following \cite[Cor. 3.2]{Dolg}

\begin{thm}\label{m-polarized moduli space} (1) There exists a fine  moduli space $\mathfrak{M} _M$ of marked $M$-polarized K3 surfaces. It is a non-separated analytic space, which is an analytic subspace of $\mathfrak{M}$. The universal family $f_M:\cU _M\to \mathfrak{M}_M$ is the restriction of the universal family $f:\cU \to \mathfrak{M}$.

(2) The obvious period map $P_M:\mathfrak{M}_M\to \Omega _M$ is analytic, etale and surjective.

(3) The diagram of analytic maps
$$\begin{array}{rcr}
\mathfrak{M}_M & \hookrightarrow & \mathfrak{M}\\
P_M \downarrow & & P\downarrow \\
\Omega _M & \hookrightarrow & \Omega
\end{array}
$$
commutes.

(4) Let $\mathfrak{M}^a_M\subset \mathfrak{M}_M$ denote the subspace parametrizing marked ample $M$-polarized K3 surfaces. Then the restriction to $\mathfrak{M}^a_M$ of the family  $f_M:\cU _M\to \mathfrak{M}_M$ is the universal family $f^a_M:\cU ^a _M\to \mathfrak{M}^a_M$ of marked ample $M$-polarized K3 surfaces. The subset $\mathfrak{M}^a_M\subset \mathfrak{M}_M$ is open. The restriction of the period map $P_M$ is an isomorphism
$$P_M^a:\mathfrak{M}^a_M\stackrel{\sim}{\to} \Omega _M^0$$
In particular, the space $\mathfrak{M}^a_M$ has 2 connected components.
\end{thm}

Consider the group
$$\Gamma _M=\{\sigma \in O(L)\ \vert \ \sigma (m)=m\ \text{for all $m\in a(M)$}\}$$
This group acts on the space $\mathfrak{M}_M$ in the obvious way: $\sigma (X,j,u)=(X,j,\sigma \cdot u)$. It preserves the subspace $\mathfrak{M}^a_M$

Notice that the above concepts of a marked $M$-polarized K3 surface, the moduli space $\mathfrak{M}_M$, and the group $\Gamma _M$ only make sense after we have made a choice of a primitive lattice embedding $a: M\hookrightarrow L$.
As in \cite{Dolg} we consider the following condition on the lattice $M$:

\medskip

(U) For any two primitive embeddings $a_1, a_2 : M \hookrightarrow L$, there exists an isometry $\sigma : L \to L$ such that
$\sigma \cdot a_1= a_2$.

\begin{lemma} \label{lemma coarse moduli space} Assuming condition (U), the quotient space $\mathfrak{M}_M/\Gamma _M$ is the coarse moduli space of $M$-polarized K3 surfaces. Hence also $\mathfrak{M}^a_M/\Gamma _M$ is the coarse moduli space of ample $M$-polarized K3 surfaces.
\end{lemma}

\begin{proof} The assumption (U) means that any $M$-polarized K3 surface $(X,j)$ can be complemented to a marked $M$-polarized K3 surface. Indeed, choose any lattice isomorphism $u:H^2(X,\bbZ )\to L$. Then by condition (U) there exists an automorphism $\sigma :L\to L$, such that $(X,j,\sigma \cdot u)$ is a marked $M$-polarized K3 surface. In particular, the forgetful map
$(X,j,u) \mapsto (X,j)$
from isomorphism classes of marked $M$-polarized K3 surfaces to isomorphism classes of $M$-polarized K3 surfaces is surjective. Obviously, the group $\Gamma _M$ acts on the fibers of this map.

It remains to show that given marked $M$-polarized K3 surfaces $(X,j,u)$ and $(X',j',u')$ such that $(X,j)\simeq (X',j')$, there exists a $\tau \in \Gamma _M$, such that
$(X,j,u)\simeq (X',j',\tau \cdot u')$. So assume that there exists an isomorphism $\phi :X\to X'$ such that
$$j=\phi ^* \cdot j':M\to Pic(X)\subset H^2(X,\bbZ)$$
Then the automorphism $\tau :=u\cdot \phi ^*\cdot (u')^{-1}:L\to L$ is the identity on $a(M)$, i.e. $\tau \in \Gamma _M$, which means that $\phi$ induces an isomorphism $(X,j,u)\simeq (X',j',\tau \cdot u')$.
\end{proof}

Let $\mathfrak{M}^{a,0}_M\subset \mathfrak{M}^a_M$ be one of the connected components (Theorem \ref{m-polarized moduli space}) and let $f^{a,0}_M:\cU ^{a,0}_M\to \mathfrak{M}^{a,0}_M$ be the restriction of the universal family $f^a_M:\cU ^a_M\to \mathfrak{M}^a_M$. Let $\Gamma _M^0\subset  \Gamma _M$be the stabilizer of the component $\mathfrak{M}^{a,0}_M$. So the index of $\Gamma _M^0$ in  $\Gamma _M$ is at most 2.

\begin{remark}\label{gamma m is a subgr of o(n)}
We note for future reference that $\Gamma _M$ (and hence also $\Gamma ^0_M$) is a subgroup of finite index in the orthogonal group $O(N)$ ($N=M^\perp _L$) \cite[Prop. 3.3]{Dolg}.
\end{remark}

\begin{prop} \label{mon group for m-polarized k3} Assume that condition (U) holds. Consider $f^{a,0}_M:\cU ^{a,0}_M\to \mathfrak{M}^{a,0}_M$ as a family of {\bfseries unmarked} ample $M$-polarized K3 surfaces. Then its monodromy group is isomorphic  to $\Gamma ^0 _M$ up to finite groups (Definition \ref{def mon}).
\end{prop}

\begin{proof} The marking defines a canonical trivialization of the local system ${\bf R}^\bullet f^{a,0}_{M*}\bbQ _{\cU ^{a,0}_M}$ on $\mathfrak{M}^{a,0}_M$ and clearly the $\Gamma ^0_M$-action on
$\bbQ \oplus L_\bbQ \oplus \bbQ =H^0(\mathfrak{M}^{a,0}_M,{\bf R}^\bullet f^{a,0}_{M*}\bbQ _{\cU^{a,0} _M})$ is effective. Since $\mathfrak{M}^{a,0} _M/\Gamma ^0_M$ is the coarse moduli space of ample $M$-polarized K3 surfaces appearing in the family $\cU ^{a,0}_M$ (Lemma
\ref{lemma coarse moduli space}), it remains to show that the $\Gamma ^{0} _M$-action on $\mathfrak{M}^{a,0}_M$ is generically free modulo a finite kernel.

As in the proof of Corollary \ref{monodr for general k3} it is enough to show that $\Gamma _M$ acts generically free modulo a finite kernel on the period domain $\Omega _M=\Omega \cap \bbP (N_\bbC )$ (Theorem \ref{m-polarized moduli space}). Since $\Gamma _M\subset O(N)$ it suffices to analyze the $O(N)$-action on $\Omega _M$.
Applying a version of Lemma \ref{eigenspace for g} with $O(N)$ and $\Omega _M$ instead of $O(L)$ and $\Omega $, we find that $O(N)$ acts generically free on $\Omega _M$ modulo its center $\{\pm 1\}_{O(N)}$.
It follows that $\Gamma _M$ acts generically free on $\Omega _M$ modulo its central subgroup $\Gamma _M\cap \{\pm 1\}_{O(N)}$. Hence $\Gamma _M$ acts on $\mathfrak{M}_M$ either generically free or generically free modulo $\Gamma _M\cap \{\pm 1\}_{O(N)}$.
\end{proof}

\subsection{Mirror symmetric families of lattice polarized
K3 surfaces} Let $M$ be a lattice as above with a fixed primitive embedding of lattices $M\hookrightarrow L$. We will identify $M$ with its image in $L$. We call a primitive sublattice $M^{\vee}\subset L$ a {\it mirror dual} of $M$ if there is  a direct sum decomposition
$$L=M\oplus U \oplus M^{\vee}$$
The signature of $M^\vee $ is $(1,18-s)$ (if the signature of $M$ is $(1,s)$). It is clear that $M=M^{\vee \vee}$. (Our definition of a mirror dual sublattice is a somewhat simplified version of \cite{Dolg}).

\begin{defn} In the above notation we consider the universal families
$f^{a,0}_M:\cU^{a,0} _M\to \mathfrak{M}^{a,0}_M$ and $f^{a,0}_{M^\vee}:\cU^{a,0} _{M^\vee}\to \mathfrak{M}^{a,0}_{M^\vee}$ as {\bfseries mirror symmetric families} of ample lattice polarized K3 surfaces.
\end{defn}

Our main result is the following.

\begin{thm}\label{main thm for k3} In the above notation assume that the lattices $M$ and $M^\vee$ satisfies condition (U). Then the groups $G^{mon}(\cU^{a,0} _{M^\vee})$ and $G^{eq}(\cU^{a,0} _{M})$ are isomorphic up to finite groups. That is, Conjecture \ref{subs main conj} holds for mirror symmetric families of ample lattice polarized K3 surfaces.
\end{thm}

\subsection{Proof of Theorem \ref{main thm for k3}} The proof will take several steps.

By assumption we have sublattices $M,M^\vee \subset L$ of signatures $(1,s)$ and $(1,18-s)$ respectively that satisty
$$(M^\vee)^\perp _L=M\oplus U$$
By Proposition \ref{mon group for m-polarized k3} the monodromy group of the family $\cU ^{a,0}_{M^\vee}$ is isomorphic up to finite groups to
$$\Gamma _{M^\vee}=\{\sigma \in O(L)\ \vert \ \sigma (m)=m\ \text{for all $m\in M^\vee$}\}$$
We have the natural injective homomorphism $\Gamma _{M^\vee}\hookrightarrow  O((M^\vee)^\perp _L)=O(M\oplus U)$ and by Remark \ref{gamma m is a subgr of o(n)} the image is a subgroup of finite index. Therefore $G^{mon}(\cU ^{a,0} _{M^\vee})$ is isomorphic up to finite groups to the group $O(M\oplus U)$, and it suffices to prove the following proposition.

\begin{prop} \label{general deq for m-polarized k3} For a general  $M$-polarized K3 surface $X$ the group $G^{eq}(X)$ is isomorphic up to finite groups to $O(M\oplus U)$.
\end{prop}

\begin{proof} By Remark \ref{remark polarized means projective} any $M$-polarized K3 surface is projective. For any projective K3 surface $Y$ the group $G^{eq}(Y)$ is well known: it is a subgroup of the group $O_{hdg}(\tilde{H}^2(Y,\bbZ))$ (Section \ref{prelim on hyperk}) of index at most two \cite[Ch. 10]{Huyb}.
So for the proof of Proposition \ref{general deq for m-polarized k3}
it remains to show that for a general $M$-polarized K3 surface $X$ the groups $O(M\oplus U)$ and $O_{hdg}(\tilde{H}^2(X,\bbZ))$ are isomorphic up to finite groups.

\begin{lemma}\label{for gen k3 equal m=pic} For a general $M$-polarized K3 surface $X$ we have the equality $M=Pic(X)$.
\end{lemma}

\begin{proof} It follows from the assumption (U) on the lattice $M$ that any $M$-polarized K3 surface can be complemented to a marked $M$-polarized K3 surface (see proof of Lemma \ref{lemma coarse moduli space}). So it suffices to prove the equality $M=Pic(X)$ for a general {\it marked} $M$-polarized K3 surface $X$. Given such a surface $X$ we consider the corresponding point $[X]\in \mathfrak{M}_M$ and its image $P_M([X])\in \Omega _M$ in the period domain. If $l\in Pic(X)$, then $P_M([X])\in \Omega _M\cap l^\perp _L$ and unless $l\in M$, this intersection $\Omega _M\cap l^\perp _L$ is a proper analytic subset of $\Omega _M$. So $P_M^{-1}(\Omega _M\cap l^\perp _L)$ is a proper analytic subset of $\mathfrak{M}_M$ (Theorem \ref{m-polarized moduli space}). Because $\mathfrak{M}_M$ is a Baire space it is not a countable union of nowhere dense subsets. This proves the lemma.
\end{proof}

For an $M$-polarized K3 surface $(X,j)$ we will identify $M$ with its image $j(M)\subset Pic(X)$. Consider the extension of the sublattice $M\subset H^2(X,\bbZ)$ to the primitive sublattice
$$\tilde{M}:=M\oplus H^0(X,\bbZ)\oplus H^4(X,\bbZ)$$
of $\tilde{H}^2(X,\bbZ)$. Then abstractly $\tilde{M}\simeq M\oplus U$. In particular $O(\tilde{M})=O(M\oplus U)$. Assuming that $X$ is general, by Lemma \ref{for gen k3 equal m=pic} we may assume that $M=Pic(X)$. Then the group $O(\tilde{M})$ and $O_{hdg}(\tilde{H}^2(X,\bbZ))$ are isomorphic up to finite groups as is shown in the proof of Lemma \ref{easy useful lemma}. This proves Proposition \ref{general deq for m-polarized k3} and Theorem \ref{main thm for k3}.
\end{proof}

\section{Conjecture \ref{subs main conj} for dual families of abelian varieties}\label{sect for ab var}

In \cite{GLO} there was defined a notion of {\it mirror symmetry for algebraic pairs} (see Definition \ref{defn of mirror} below). An {\it algebraic pair} $(A,\omega )$ consists of an abelian variety $A$ and an element $\omega$ of the complexified ample cone of $A$ (Definition \ref{def of alg pair}). Building on this work we define the notion of mirror symmetric {\it families of abelian varieties} (Definition \ref{defn of mirror sym for ab var}).
Then we prove Conjecture \ref{subs main conj} for such families. We start by recalling some relevant facts about abelian varieties.

\subsection{Complex tori and abelian varieties}\cite{Mu1}, \cite{BirLa}, \cite{GLO}.

\subsubsection{} Let $\Gamma \simeq \bbZ ^{2n}$ be a lattice, $V=\Gamma \otimes \bbR \simeq \bbR ^{2n}$ and $J\in \End _{\bbR}(V)$, s.t. $J^2=-1$.
(Here {\it a lattice} means a discrete subgroup of finite covolume).
That is $J$ is a complex structure on $V$. This way we obtain an $n$-dimensional complex torus
$$A=(V/\Gamma ,J).$$
Note the canonical isomorphisms
$$\Gamma =H_1(A,\bbZ),\quad V=H_1(A,\bbR).$$
Sometimes we will write $\Gamma _A,V_A,J_A$.

Given another complex torus $B=(V_B/\Gamma _B,J_B)$, the group $\Hom (A,B)$ consists of homomorphisms $f:\Gamma _A\to \Gamma _B$ such that
$$J_B\cdot f_{\bbR}=f_{\bbR}\cdot J_A:V_A\to V_B$$
Thus the abelian group $\Hom (A,B)$ can be considered as a subgroup of $\Hom (\Gamma _A,\Gamma _B)$.

\subsubsection{} One has the dual torus $\hat{A}$ defined as follows. Put $\Gamma ^*=\Hom _{\bbZ}(\Gamma ,\bbZ),$ $V^*=\Gamma ^*\otimes \bbR=\Hom (V,\bbR)$ and $\hat{J}:V^*\to V^*$, s.t. $(\hat{J}w)(v)=w(-Jv)$ for $v\in V,w\in V^*$. Then by definition
$$\hat{A}=(V^*/\Gamma ^*,\hat{J}).$$

\subsubsection{}\label{pic} Denote by $\Pic _A$ the Picard group of $A$. Let $\Pic ^0_A\subset \Pic _A$ be the subgroup of line bundles with the trivial Chern class. It has a natural structure of a complex torus. Moreover, there exists a natural isomorphism of complex tori
$$\hat{A}\simeq \Pic ^0_A.$$
Every line bundle $L$ on $A$ defines a morphism $\phi _L :A\to \hat{A}$ by the formula
$$\phi _L(a)=T^*_aL\otimes L^{-1}.$$
(Here $T_a:A\to A$ is the translation by $a$). We have $\phi _L=0$ iff $L\in \Pic ^0_A$, and $\phi _L$ is an isogeny if $L$ is ample. Thus the correspondence $L\mapsto \phi _L$ identifies the N\'eron-Severi group $NS_A:=\Pic _A/\Pic ^0_A$ as a subgroup in $\Hom (A,\hat{A})$. Also $NS_A$ is naturally a subgroup of $H^2(A,\bbZ)$: to a line bundle $L$ there corresponds its first Chern class, which can be considered as a skew-symmetric bilinear form on $\Gamma $. Put $c_1(L)=c$. Then the morphism $\phi _L$ is given by the map
$$V_A\to V_{\hat{A}},\quad v\mapsto c(v,\cdot).$$

We will identify $NS_A$ either as a subgroup of $\Hom (A,\hat{A})$ or $\Hom (\Gamma _A,\Gamma _{\hat{A}})$ or as a set of (intergal) skew-symmetric forms $c$ on $\Gamma _A$ such that the extension $c_{\bbR}$ on $V_A$ is $J$-invariant.

\subsubsection{}
Given a morphism of complex tori $f:A\to B$, the dual morphism $\hat{f} :\hat{B}\to \hat{A}$ is defined.

The double dual torus $\Hat{\Hat{A}}$ is naturally identified with $A$ by means of the Poincar\'e line bundle on $A\times \hat{A}$ and $\hat{A}\times \Hat{\Hat{A}}$. So given a morphism $f:A\to \hat{A}$ we will consider $\hat{f}:\Hat{\Hat{A}}\to \hat{A}$ as a morphism from $A$ to $\hat{A}$ again. Then for $L\in NS_A$ we have $\hat{\phi} _L=\phi _L$.

\subsubsection{} Consider the lattice $\Lambda =\Lambda _A:=\Gamma _A\oplus \Gamma ^*_A$ with the canonical symmetric bilinear form
$Q:\Lambda \times \Lambda \to \bbZ$ defines as follows
$$Q((a_1,b_1),(a_2,b_2))=b_1(a_2)+b_2(a_1).$$
Let $O(\Lambda ,Q)\subset GL(\Lambda)$ be the corresponding orthogonal group. It is equal to
$$O(\Lambda ,Q)=\left\{g=\left(\begin{array}{cc} a & b\\
c & d\end{array}\right)\in
\left(\begin{array}{ll}
\Hom(\Gamma ,\Gamma) & \Hom (\Gamma ^*,\Gamma )\\
\Hom (\Gamma ,\Gamma ^*) & \Hom (\Gamma ^*,\Gamma ^*)\end{array}\right) \Biggl|\ g^{-1}=\left(\begin{array}{rr}
\hat{d} & -\hat{b}\\
-\hat{c} & \hat{a}\end{array}\right)\right\}$$
where $\Gamma =\Gamma _A$.

Notice that if $A=(V/\Gamma _A,J_A)$ is a complex torus, then the complex structure $J_{A\times \hat{A}}$ of the product $A\times \hat{A}=(\Lambda _\bbR/\Lambda ,J_{A\times \hat{A}})$ belongs to the special orthogonal group $SO(\Lambda _\bbR, Q_\bbR)$.

\subsubsection{} A complex torus $A=(V/\Gamma ,J)$ is algebraic, i.e. an abelian variety, iff there exists $c\in NS_A$ such that the symmetric bilinear form $c_{\bbR}(J\cdot ,\cdot)$ on $V$ is positive definite. If a line bundle $L\in \Pic _A$ is ample then the induced map
$$\phi _L:\Gamma _{A,\bbQ}\to \Gamma _{\hat{A},\bbQ}$$
is an isomorphism.

We will only be interested in complex tori which are abelian varieties.

\subsection{Hodge group of an abelian variety} Let $W$ be a finite dimensional $\bbQ$-vector space and $J\in \End (W_\bbR)$ a complex structure, i.e. $J^2=-1$. This defines an embedding of $\bbR$-algebras $\bbC \subset \End (W_\bbR)$ and in particular an inclusion of groups $h:S^1 \hookrightarrow Aut (W_\bbR)$ such that $h(\sqrt{-1})=J$.

\begin{defn} \label{def of hodge gp} The {\rm Hodge group} of the complex structure $J$ is the smallest algebraic $\bbQ$-subgroup $H\subset Aut(W)$ such that $h(S^1)\subset H(\bbR)$. We denote it by $\langle J\rangle _\bbQ$.
If $A=\langle V_A/\Gamma _A,J_A\rangle $ is an abelian variety, the Hodge group $\langle J_A\rangle _\bbQ\subset Gl (\Gamma _{A,\bbQ})$ is also denoted by $Hdg _A$.
\end{defn}

\begin{remark} Since the Lie group $S^1$ is connected, so is the algebraic $\bbQ$-group $\langle J\rangle _\bbQ$.
\end{remark}

We have canonical identifications
$$Hdg _{A}=Hdg _{\hat{A}}=Hdg _{A\times \hat{A}}$$
Depending on the context we may view $Hdg _{A}$ as a subgroup of $Gl(\Gamma _{A,\bbQ})$ or $Gl(\Gamma _{\hat{A},\bbQ})$ or
$SO(\Lambda _\bbQ ,Q_\bbQ)$. (Indeed, by construction $J_{A\times \hat{A}}\in O(\Lambda _\bbR ,Q_\bbR)$ and so $Hdg_{A\times \hat{A}}\subset O(\Lambda _\bbQ ,Q_\bbQ)$; hence $Hdg_{A\times \hat{A}}\subset SO(\Lambda _\bbQ ,Q_\bbQ)$ because $Hdg_{A\times \hat{A}}$ is connected.)

\subsubsection{} The following facts about the group $Hdg _A$ are known (\cite{Mu2}, \cite{Del1})

\begin{thm} \label{facts about hodge} Assume that $A$ is an abelian variety.

(a) $Hdg _A$ is a connected reductive algebraic $\bbQ$-group without simple factors of exceptional type.

(b) The adjoint action of $J_A$ on the Lie group $Hdg _A(\bbR)^0$ is a Cartan involution, i.e. it is an involution whose fixed subgroup is a maximal compact subgroup $K$.

(c) The symmetric space $Hdg _A(\bbR)^0/K$ is of Hermitian type.
\end{thm}

\subsection{Derived category of an abelian variety}

Let $A$ be an abelian variety. In this case the action of the group $AutD^b(A)$ preserves the {\it integral} cohomology of $A$, i.e. we have the homomorphism
$$\rho _A:AutD^b(A)\to Gl(H^\bullet (A,\bbZ))$$
(In [GLO] the image of $\rho _A$ is denoted $Spin(A)$, but here we denote it $G^{eq}(A)$.) This group tends to be big and there exists a precise description of this group in terms of the {\it Mukai-Polishchuk} group $U(A)$. Let us recall it.

\begin{defn} For an abelian variety $A$ put
$$U(A)=\left\{g=\left(\begin{array}{cc} a & b\\
c & d\end{array}\right)\in
\left(\begin{array}{ll}
\End(A) & \Hom (\hat{A},A)\\
\Hom (A,\hat{A}) & \End (\hat{A})\end{array}\right) \Biggl|\ g^{-1}=\left(\begin{array}{rr}
\hat{d} & -\hat{b}\\
-\hat{c} & \hat{a}\end{array}\right)\right\}$$
\end{defn}

So by definition we have $U(A)=Aut(A\times \hat{A})\cap O(\Lambda ,Q)$, which also equals $Aut(A\times \hat{A})\cap SO(\Lambda ,Q))$, because elements of $Aut(A\times \hat{A})$ have positive determinant (as they preserve the complex structure on $V_A\oplus V_{\hat{A}}$).

\subsubsection{} For us the group $U(A)$ is important because of the following facts.

\begin{prop} There exists a natural exact sequence of groups
$$0\to \bbZ \times A\times \hat{A}\to Auteq(D^b(A))\to U(A)\to 1$$
\end{prop}

The homomorphism $\rho _A:Aut (D^b(A))\to Gl(H^\bullet (A,\bbZ))$ almost factors through the group $U(A)$. Namely, we have the exact sequence of groups
$$0\to \bbZ /2\bbZ \to G^{eq}(A)\to U(A)\to 1$$

\begin{remark} \label{remark commensur} It follows that the groups $G^{eq}(A)$ and $U(A)$ are isomorphic up to finite groups.
\end{remark}

As explained in \cite{GLO} the group $SO(\Lambda ,Q)$ does not act on the space
$H^\bullet (A,\bbZ)$, but its double cover does. Namely, there is a discrete group $Spin(\Lambda ,Q)$ and an exact sequence
$$0\to \bbZ /2\bbZ \to Spin (\Lambda ,Q)\to SO(\Lambda ,Q)$$
The group $Spin (\Lambda ,Q)$ acts on $H^\bullet (A,\bbZ)=\Lambda ^\bullet \Gamma ^* _A$ via the {\it spinorial} representation. Moreover we have the commutative diagram
$$\begin{array}{ccc}
G^{eq}(A)& \hookrightarrow & Spin (\Lambda ,Q)\\
\downarrow & & \downarrow \\
U(A) & \hookrightarrow & SO(\Lambda ,Q)
\end{array}
$$
and the action of $G^{eq}(A)$ on $H^\bullet (A,\bbZ)$ is the restriction of the spinorial representation of $Spin (\Lambda ,Q)$.

\subsection{The algebraic $\bbQ$-group $\U _{A,\bbQ}$}

Let $A$ be an abelian variety and let $Aut _{\bbQ}(A\times \hat{A})$ be the group of invertible elements in $\End ^0(A\times \hat{A}):=\End (A\times \hat{A})\otimes \bbQ$. Define the algebraic $\bbQ$-group $\U _{A,\bbQ}$ as follows
$$\U _{A,\bbQ}=\left\{g=\left(\begin{array}{cc} a & b\\
c & d\end{array}\right)\in
Aut _{\bbQ}(A\times \hat{A})\ \Biggl|\ g^{-1}=\left(\begin{array}{rr}
\hat{d} & -\hat{b}\\
-\hat{c} & \hat{a}\end{array}\right)\right\}$$
So $\U _{A,\bbQ}=Aut _{\bbQ}(A\times \hat{A})\cap O(\Lambda _\bbQ,Q_\bbQ)=Aut _{\bbQ}(A\times \hat{A})\cap SO(\Lambda _\bbQ,Q_\bbQ)$ and $U(A)$ is the arithmetic subgroup of $\U _{A,\bbQ}$ consisting of elements that preserve the lattice $\Lambda$.

\begin{remark} \label{centralizer} Note that the algebraic $\bbQ$-group $\U _{A,\bbQ}$ is the centralizer in $SO(\Lambda _\bbQ,Q_\bbQ)$ of the group $Hdg _A$.
\end{remark}

The group $\U_{A,\bbQ}$ was introduced and studied in \cite{Pol}.

\begin{thm} \label{list of prop of nonconn} \cite{Pol} For an abelian variety $A$ the group $\U _{A,\bbQ}$ is reductive.
\end{thm}

\subsection{The algebraic $\bbQ$-group $U(A)_\bbQ$} It will be convenient for us to consider a slightly smaller algebraic $\bbQ$-group. Namely, first consider the $\bbQ$-Zariski closure in
$\U _{A,\bbQ}$ of its arithmetic subgroup $U(A)$. (In \cite{GLO} this group was denoted $\overline{U(A)}$.) Let $U(A)_\bbQ :=\overline{U(A)}^0$ be its connected component. This is the algebraic $\bbQ$-group, that we will be interested in.

The main properties of this group are summarized in the following proposition.

\begin{prop} Let $A$ be an abelian variety.

(1) The group $U(A)_\bbQ$ is semisimple.

(2) The semi-simple Lie group $U(A)_\bbQ (\bbR)^0$ consists of all the non-compact factors of the reductive Lie group $\U _{A,\bbQ}(\bbR)^0$.

(3) The arithmetic subgroup $U(A)^0:=U(A)\cap U(A)_\bbQ (\bbR)^0$ of
$U(A)_\bbQ (\bbR)^0$ is Zariski dense.
\end{prop}

\begin{proof}  (1) \cite[5.3.5]{GLO}, (2) \cite[7.2.1]{GLO},
(3) \cite[Thm.1]{Bor}.
\end{proof}

\begin{remark} \label{remark commeas} The subgroup $U(A)^0\subset U(A)$ has finite index. Since the groups $U(A)$ and $G^{eq}(A)$ are isomorphic up to finite groups, so are $U(A)^0$ and $G^{eq}(A)$ (Remark \ref{remark commensur}).
\end{remark}

\begin{remark} \label{coincidence of groups} Let $A$ and $B$ be abelian varieties with an identification of lattices $\Lambda _A=\Lambda _B=\Lambda$ which is compatible with the form $Q$. Assume that under this identification we have the equality of algebraic groups $\U _{A,\bbQ}=\U _{B,\bbQ}$. Then clearly $U(A)=U(B)$ and hence also $U(A)_\bbQ=U(B)_\bbQ$, $U(A)^0=U(B)^0$, etc.. (The equality
$\U _{A,\bbQ}=\U _{B,\bbQ}$ holds for example when $Hdg _A=Hdg _B$ as subgroups in $SO(\Lambda _\bbQ,Q_\bbQ)$.)
\end{remark}

\subsection{Action of the Lie group $\U _{A,\bbQ}(\bbR)$ on a Siegel domain}
Let $A$ be an abelian variety. Let us define a rational (i.e. not everywhere defined) action of the Lie group $\U _{A,\bbQ}(\bbR)$ on the complex space $NS_{A, \bbC} \subset \Hom (A,\hat{A})\otimes \bbC$:
$$\begin{array}{cc}
\left(\begin{array}{cc}
a & b\\
c & d \end{array}\right)\omega :=(c+d\omega)(a+b\omega)^{-1},\\
\left(\begin{array}{cc}
a & b\\
c & d \end{array}\right)\in \U _{A,\bbQ}(\bbR),\quad \omega \in NS_{A,\bbC}\subset \Hom (A,\hat{A})\otimes \bbC.
\end{array}
$$
Here the multiplication is understood as composition of maps.

$NS_{A,\bbC}$ contains a Siegel domain of the first kind \cite{Pjat} on which this action is well defined. Namely, let $C_A^a\subset NS_{A,\bbR}$ be the ample cone of $A$, which is defined as the set of $\bbR ^+$-linear combinations of ample classes in $NS_A$. It is an open subset in $NS_{A,\bbR}$. Consider the complexified ample cone
$$
C_A:=NS_{A,\bbR}+ iC^a_A\subset NS_{A,\bbC}
$$
(Note that in \cite{GLO} $C_A$ denotes the bigger set $NS_{A,\bbR}\pm iC^a_A$ which has two connected components.)

\begin{thm} \label{action on z domain} Let $A$ be an abelian variety.

(1) The action of $\U _{A,\bbQ}(\bbR)$ on $C_A$ is well defined and is transitive.

(2) The stabilizer   of a point in $C_A$ is a maximal compact subgroup in $\U _{A,\bbQ}(\bbR)$.

(3) The action of the subgroup $U(A)_\bbQ (\bbR)^0\subset \U _{A,\bbQ}(\bbR)$ on $C_A$ is also transitive and the stabilizer of a point is a maximal compact subgroup of $U(A)_\bbQ (\bbR)^0$. Hence $C_A$ is the  Hermitian symmetric space for the semi-simple Lie group $U(A)_\bbQ (\bbR)^0$.
\end{thm}

\begin{proof} \cite[8.2, 8.3]{GLO}.
\end{proof}

\subsection{Mirror symmetry for algebraic pairs $(A,\omega)$ following [GLO]}

\begin{defn} \label{def of alg pair} An {\bfseries algebraic pair} is a pair $(A,\omega)$, where $A$ is an abelian variety and $\omega \in C_A$.
\end{defn}

Let us recall the notion of mirror symmetry for algebraic pairs from [GLO].
Consider the $\U _{A,\bbQ}(\bbR)$-action on $C_A$ defined above. Given $\omega =\phi _1+i\phi _2\in C_A$, define
\begin{equation}\label{formula for iomega} \begin{array}{rcl}
I_\omega & := & \left(\begin{array}{cc}
\phi ^{-1}_2\phi _1 & -\phi _2^{-1}\\
\phi _2+\phi _1\phi _2^{-1}\phi _1 & -\phi _1\phi _2^{-1}
\end{array}\right)\\
 & = & \left(\begin{array}{cc}1 & 0\\\phi _1 & 1\end{array}\right)
 \left(\begin{array}{cc}0 & -\phi _2^{-1}\\\phi _2 & 0\end{array}\right)
 \left(\begin{array}{cc}1 & 0\\-\phi _1 & 1\end{array}\right)\in \U _{A,\bbQ}(\bbR)
 \end{array}
\end{equation}

The following proposition together with Theorem \ref{action on z domain} should be compared with Theorem \ref{facts about hodge}.

\begin{prop} \label{properties of action} Consider the $U(A)_\bbQ(\bbR)^0$-action on $C_A$
(Theorem \ref{action on z domain}).
Let $\omega \in C_A$ and let $K_\omega \subset U(A)_\bbQ(\bbR)^0$ be its stabilizer. Then the following holds.

(1) The operator $I_\omega$ belongs to the center of $K_\omega$ (in particular $I_\omega \in U(A)_\bbQ(\bbR)^0$) and the adjoint action of $I_\omega$ on $U(A)_\bbQ (\bbR)$ is the Cartan involution corresponding to $K_\omega$.

(2) We have $I_\omega ^2=-1$, hence $I_\omega$ defines a complex structure on $\Lambda _\bbR$.

(3) The correspondence $\omega \mapsto I_\omega $ is injective.
\end{prop}

\begin{proof} \cite[8.4.1]{GLO}.
\end{proof}

\begin{remark} \label{remark on q envelop} Let $\omega \in C_A$. It follows from Proposition \ref{properties of action} that there is an inclusion of algebraic $\bbQ$-groups $\langle I_\omega \rangle _\bbQ \subset U(A)_\bbQ$ (Definition \ref{def of hodge gp}).
\end{remark}

Consider the real vector space $V_A\oplus V_{\hat{A}}$. It has the complex structure $J_{A\times \hat{A}}$. Since the group $U(A)_\bbQ (\bbR)^0$ acts on $V_A\oplus V_{\hat{A}}$, for each $\omega \in C_A$ the operator $I_\omega$ defines another complex structure on $V_A\oplus V_{\hat{A}}$. These complex structures commute.

\begin{defn} \label{defn of mirror}\cite[9.2]{GLO} Algebraic pairs $(A,\omega _A)$ and $(B,\omega _B)$ are {\bfseries mirror symmetric} if there is an isomorphism of lattices
$$\alpha :\Lambda _A\stackrel{\sim}{\to}
\Lambda _B$$
which identifies the bilinear forms $Q_A$ and $Q_B$ and satisfies the following conditions
\begin{equation}\label{cond of mirror}\begin{array}{lll}
\alpha _{\bbR}\cdot J_{A\times \hat{A}} & = & I_{\omega _B}\cdot \alpha _{\bbR},\\
\alpha _{\bbR}\cdot I_{\omega _A} & = & J_{B\times \hat{B}}\cdot \alpha _{\bbR}
\end{array}
\end{equation}
\end{defn}

Let algebraic pairs $(A,\omega _A)$ and $(B,\omega _B)$ be mirror symmetric. We may assume that $\Lambda =\Lambda _A=\Lambda _B$ and $\alpha =\id$. Then we obtain the inclusions of algebraic $\bbQ$-groups
$$Hdg _A\subset U(B)_\bbQ,\quad \text{and}\quad
Hdg _B\subset U(A)_\bbQ$$
as subgroups of $SO(\Lambda _\bbQ,Q_\bbQ)$ (Remark \ref{remark on q envelop}).

\subsection{How to find a mirror symmetric pair}\label{how to find?}

Let $(A,\omega _A)$ be an algebraic pair. As explained above we obtain two commuting complex structures on the real vector space
$V_A\oplus V_{\hat{A}}$: $J_{A\times \hat{A}}$ and $I_{\omega _A}$.
To find a mirror pair $(B,\omega _B)$ we need the following:
\begin{itemize}
\item[(1)] Find a $Q_A$-isotropic decomposition $\Lambda _A=\Gamma _1\oplus \Gamma _2$ such that the vector subspaces $\Gamma _{1\bbR},\Gamma _{2\bbR}
    \subset V_A\oplus V_{\hat{A}}$ are $I_{\omega _A}$-invariant. This will give a complex torus $B=(\Gamma _{1\bbR}/\Gamma _1,I_{\omega _A}\vert _{\Gamma _{1\bbR}})$ and the dual torus $\hat{B}=(\Gamma _{2\bbR}/\Gamma _2,I_{\omega _A}\vert_{\Gamma _{2\bbR}})$ and an isomorphism of complex tori
    $$((V_A\oplus V_{\hat{A}})/(\Gamma _A\oplus \Gamma _{\hat{A}}),I_{\omega _A})\simeq B\times \hat{B}$$
\item[(2)] Show that $C_B\neq \emptyset$ (i.e. $B$ is an abelian variety) and there exists $\omega _B\in C_B$ such that the operator $I_{\omega _B}$ on $V_B\oplus V_{\hat{B}}=V_A\oplus V_{\hat{A}}$ coincides with $J_{A\times \hat{A}}$. This will show that the algebraic pairs $(A,\omega _A)$ and $B,\omega _B)$ are mirror symmetric (take $\alpha =\id$).
    \end{itemize}

Actually if (1) is achieved, then (2) is automatic \cite[9.4.6]{GLO}.

\begin{remark} (1) It is not true that for every algebraic pair $(A,\omega _A)$ there exists a mirror symmetric pair. The problem may occur if the group $U(A)_\bbQ$ is too big and $\omega _A\in C_A$ is chosen too general \cite[9.5.1]{GLO}. But for every abelian variety $A$ there exists an element $\omega \in C_A$ such that the pair $(A,\omega )$ has a mirror symmetric pair \cite[10.4.3]{GLO}.

(2) The mirror pair, if exists, may not be unique. However for a given pair $(A,\omega _A)$ the collection of isomorphism classes of abelian varieties $B$ for which there exists $\omega _B\in C_B$ such that the pairs $(A,\omega _A)$ and $(B,\omega _B)$ are mirror symmetric is finite \cite[9.2.3]{GLO}.
  \end{remark}

\subsection{A useful lemma} We recall a result from \cite{GLO} that will be useful later. Let $A$ be an abelian variety, $I\in U(A)_\bbQ (\bbR)^0$ such that $I^2=-1$. Then $I$ defines a complex structure on $\Lambda _\bbR$. The complex structures $I$ and $J_{A\times \hat{A}}$ commute and preserve the bilinear form $Q_\bbR$. The operator $c:=I\cdot J_{A\times \hat{A}}$ also preserves $Q_\bbR$ and the bilinear form
$Q_\bbR (c(-),-)$ is symmetric. Denote by $E_I$ the corresponding quadratic form on $\Lambda _\bbR$.

\begin{lemma}\label{useful lemma}  The quadratic form $E_I$ is positive definite if and only if $I=I_\omega$ for some $\omega \in C_A$.
\end{lemma}

\begin{proof} This is a special case of \cite[9.4.2]{GLO}.
\end{proof}

\subsection{Perfect algebraic pairs}

\begin{defn}\label{defn perfect}  An algebraic pair $(A,\omega)$ is {\bfseries perfect} if $U(A)_\bbQ =\langle I_\omega \rangle _{\bbQ}$.
\end{defn}


\begin{lemma} \label{lemma single conj class} Let $A$ be an abelian variety.

(1) The set $\{ I_\omega \ \vert \ \omega \in C_A\}$ is a conjugacy class in $U(A)_\bbQ(\bbR)^0$.

(2) Let $\sigma $ be an automorphism of the Lie group $U(A)_\bbQ (\bbR)^0$ such that $\sigma (I_{\omega _1})=I_{\omega _2}$ for some $\omega _1,\omega _2\in C_A$. Then $\sigma $ preserves the conjugacy class $\{ I_\omega \ \vert \ \omega \in C_A\}$.
\end{lemma}

\begin{proof} (1) Notice that the center
$Z(U(A)_\bbQ (\bbR)^0)$ is discrete (as the group $U(A)_\bbQ (\bbR)^0$ is semi-simple) and the adjoint action of $I_\omega $ on the Lie algebra $Lie(U(A)_\bbQ(\bbR)^0)$ is the Cartan involution corresponding to the maximal compact subgroup $K_\omega \subset U(A)_\bbQ (\bbR)^0$ (Proposition \ref{properties of action}). Therefore for each $\omega \in C_A$ the set of elements in $U(A)_\bbQ(\bbR)^0$ whose adjoint action on the Lie algebra is the corresponding Cartan involution is the discrete set $I_\omega Z(U(A)_\bbQ(\bbR)^0)$. All maximal compact subgroups of $U(A)_\bbQ (\bbR)^0$ are conjugate \cite[Theorem 2.2]{Helg}. If $gK_\omega g^{-1}=K_{\omega ^\prime}$ for some $g\in U(A)_\bbQ (\bbR)^0$, then $gI_\omega Z(U(A)_\bbQ (\bbR)^0)g^{-1}=I_{\omega ^\prime} Z(U(A)_\bbQ(\bbR)^0)$. As $g$ belongs to the connected group
$U(A)_\bbQ(\bbR)^0$ and elements $I_\omega$ depend continuously on $\omega$, we must have $gI_\omega g^{-1}=I_{\omega ^\prime}$.

(2) This follows from (1).
\end{proof}

\begin{lemma} \label{generic perfect} Let $A$ be an abelian variety. Then there exists a subset $Z\subset C_A$ which is a countable union of proper analytic subsets such that for every $\omega \in C_A\backslash Z$ the pair $(A,\omega )$ is perfect.
\end{lemma}

\begin{proof} We claim that for some $\tau \in C_A$ the Lie group $\langle I_\tau \rangle _{\bbQ}(\bbR)$ contains the whole conjugacy class $\{ I_\omega \ \vert
\ \omega \in C_A\}$. Indeed, there exists a countable number of algebraic $\bbQ$-subgroups in $U(A)_\bbQ$. Let $B$ be one such subgroup. If $\{ I_\omega \ \vert \ \omega \in C_A\}\nsubseteq B(\bbR)$, then
\begin{equation}\label{thin intersection} B(\bbR)\cap \{ I_\omega \ \vert \ \omega \in C_A\}
\end{equation}
is a proper algebraic subset, so it is nowhere dense in $\{ I_\omega \ \vert \ \omega \in C_A\}$.
Since $\{ I_\omega \ \vert \ \omega \in C_A\}$ is a Baire space, it is not a countable union of nowhere dense subsets. Therefore there exists $\tau \in C_A$ such that $\{ I_\omega \ \vert \ \omega \in C_A\}\subset \langle I_{\tau } \rangle _{\bbQ}(\bbR)$. Fix one such $\tau \in C_A$. We claim that $\langle I_{\tau } \rangle _{\bbQ}=
U(A)_\bbQ$. Both these groups are connected, so it is enough to prove the equality of dimensions. For this it suffices to show the equality of the groups of $\bbR$-points $\langle I_{\tau } \rangle _{\bbQ}(\bbR)^0=
U(A)_\bbQ(\bbR)^0$.

First notice that $\langle I_{\tau } \rangle _{\bbQ}(\bbR)^0$ is a normal subgroup in $U(A)_\bbQ(\bbR)^0$. Indeed, for any $g\in
U(A)_\bbQ (\bbQ)$, we have
$$g(\langle I_{\tau } \rangle _{\bbQ})g^{-1}=\langle I_{g(\tau ) } \rangle _{\bbQ}\subset \langle I_{\tau } \rangle _{\bbQ}$$
Since the semisimple Lie group $U(A)_\bbQ (\bbR)^0$ has no compact factors, the group $U(A)_\bbQ (\bbQ)$ is Zariski dense in  $U(A)_\bbQ (\bbR)^0$ \cite[Thm.1]{Bor}. So $\langle I_{\tau } \rangle _{\bbQ}(\bbR)^0$ is a normal (closed) subgroup in $U(A)_\bbQ (\bbR)^0$. Thus the Lie algebra of $\langle I_{\tau } \rangle _{\bbQ}(\bbR)^0$ consists of a number of simple factors of the Lie algebra of $U(A)_\bbQ (\bbR)^0$. But the adjoint action of $\langle I_{\tau } \rangle _{\bbQ}(\bbR)^0$ on $Lie(U(A)_\bbQ (\bbR)^0)$ contains a Cartan involution. Hence
$\langle I_{\tau } \rangle _{\bbQ}(\bbR)^0=U(A)_\bbQ (\bbR)^0$.
\end{proof}

\begin{defn} For an abelian variety $A$ put $C^0_A=C_A\backslash Z$ (in the notation of Lemma \ref{generic perfect}), i.e. $C_A^0$ consists of elements $\omega$, such that $(A,\omega)$ is a perfect pair. By Lemma \ref{generic perfect} $C_A^0$ is a dense subset of $C_A$.
\end{defn}

\subsection{Mirror symmetric families of abelian varieties}

\begin{defn} Algebraic pairs $(A,\omega _A)$ and $(B,\omega _B)$ which are mirror symmetric (Definition \ref{defn of mirror}) are called {\bfseries perfectly mirror symmetric} if in addition these pairs are perfect (Definition \ref{defn perfect}).
\end{defn}

\subsection{Construction of mirror symmetric families of abelian varieties} \label{construction} Assume that the algebraic pairs $(A,\omega _A)$ and $(B,\omega _B)$ are perfectly mirror symmetric. Identify the lattices $\Lambda =\Lambda _A=\Gamma _A\oplus \Gamma _A^* $ and $\Lambda _B=\Gamma _B\oplus \Gamma _B^*$ via the isomorphism $\alpha$ (Definition \ref{defn of mirror}), and consider the algebraic $\bbQ$-groups $U(A)_\bbQ$, $U(B)_\bbQ$, $Hdg _A$, $Hdg _B$ as subgroups of the special orthogonal group $SO(\Lambda _{\bbQ},Q_\bbQ)$. By assumption the group $Hdg _A$ (resp. $Hdg _B$) is the Zariski $\bbQ$-closure of the complex structure $I_{\omega _B}$ (resp. $I_{\omega _A}$). Hence we have the equalities
\begin{equation}\label{equality of mirror groups}U(A)_\bbQ =Hdg _B,\quad U(B)_\bbQ=Hdg _A
\end{equation}
(Vice versa: if equalities \eqref{equality of mirror groups} hold for mirror symmetric pairs $(A,\omega _A)$ and $(B,\omega _B)$ then these pairs are perfect).
Since the group $Hdg _A$ preserves the subspace $\Gamma _{A,\bbQ}\subset \Lambda _\bbQ$, so does the group $U(B)_\bbQ$. It follows that for any $\eta _B\in C_B$ the complex structure $I_{\eta _B}$ on $\Lambda _\bbR$ restricts to a complex structure on
$\Gamma _{A,\bbR}\subset \Lambda _\bbR$, i.e. we get the abelian variety $A_{\eta _B}=(V_A/\Gamma _A, I_{\eta _B})$ (\ref{how to find?}).
Symmetrically, for any $\eta _A\in C_A$ we have the abelian variety $B_{\eta _A}=(V_B/\Gamma _B, I_{\eta _A})$.

In sum we obtain two families of abelian varieties
\begin{equation}\label{dual families}\cA:= \{A_{\eta _B}\ \vert \ \eta _B\in C_B\}\quad \text{and}\quad
\cB:=\{B_{\eta _A}\ \vert \ \eta _A\in C_A\}
\end{equation}
with bases $C_B$ and $C_A$ respectively.

\begin{defn} \label{defn of mirror sym for ab var} We call the families $\cA $ and $\cB$ as in \eqref{dual families} the {\bfseries mirror symmetric families of abelian varieties}.
\end{defn}

\subsection{Properties of mirror symmetric families}

If in the above notation in addition $\eta _B\in C_B^0$, then
\begin{equation}\label{equal of hodge} Hdg _{A_{\eta _B}}=U(B)_\bbQ =Hdg _A\end{equation}
Therefore $\U_{A_{\eta _B},\bbQ}=\U _{A,\bbQ}$ (Remark \ref{centralizer}) and  $U(A_{\eta _B})_\bbQ =U(A)_\bbQ$ (Remark \ref{coincidence of groups}).
So $I_{\omega _A}\in U(A_{\eta _B})_\bbQ (\bbR)^0$.
We claim that $(A_{\eta _B},\omega _A)$ is an algebraic pair, i.e. that $\omega _A\in C_{A_{\eta _B}}$. To see this we consider the operator
$$c:=I_{\omega _A} \cdot J_{A_{\eta _B}\times \hat{A}_{\eta _B}}=
J_{B\times \hat{B}}\cdot I_{\eta _B}$$
By Lemma \ref{useful lemma} $I=I_{\omega _A}$ equals $I_{\omega _{A_{\eta _B}}}$ for some $\omega _{A_{\eta _B}}\in C_{A_{\eta _B}}$ (i.e. $\omega _A\in C_{A_{\eta _B}}$) if and only if the quadratic form associated with the bilinear form $Q(c(-),-)$ is positive definite on $\Lambda _\bbR$. But the same Lemma \ref{useful lemma} implies that this form is positive definite, as $\eta _B\in C_B$.

So we obtain
perfectly symmetric pairs $(A_{\eta _B},\omega _A)$ and $(B,\eta _B)$.

Symmetrically, for any $\eta _A\in C_A^0$ we have the abelian variety $B_{\eta _A}=(V_B/\Gamma _B, I_{\eta _A})$ with
$$Hdg _{B_{\eta _A}}=Hdg _B\quad \text{and}\quad U(B_{\eta _A})_\bbQ =U(B)_\bbQ $$
and perfectly symmetric pairs $(A,\eta _A)$ and $(B_{\eta _A},\omega _B)$.

In particular we obtain the following corollary.

\begin{cor}\label{cor coincide}
(1) For any parameter $\omega \in C_A$ there is an inclusion $Hdg _{B_\omega}\subset Hdg _B$ and therefore the inclusions
$\U _{B ,\bbQ}\subset \U _{B_\omega ,\bbQ}$,
$U(B)\subset U(B_\omega )$, $U(B)_\bbQ \subset U(B_\omega )_\bbQ$
  $U(B)^0\subset U(B_\omega )^0$. Also
 $End (B)\subset End(B_\omega)$, $Aut(B)\subset Aut(B_\omega)$.
(And similarly for the abelian varieties $A_\eta$, $\eta \in C_B$.)

(2) For any $\omega \in C_A^0$ the groups $Hdg _{B_\omega}$, $\U _{B_\omega ,\bbQ}$, $U(B_\omega)$, $U(B_\omega )_\bbQ $,  $U(B_\omega )^0$,  $End(B_\omega)$, $Aut(B_\omega)$ coincide with the corresponding groups for $B$. (And similarly for the abelian varieties $A_\eta$, $\eta \in C_B^0$.)
\end{cor}

Actually also the ample cones of the abelian varieties
$\{ B_\omega \ \vert \ \omega \in C^0_A\}$ are the same. Hence the complexified cone  $C_{B_\omega}$ is independent of $\omega \in C^0_A$ (Corollary \ref{indep of ample cone}).

\begin{lemma} \label{equality of ample cones} Let $A=(V/\Gamma ,J)$ be an abelian variety, $Hdg _A=\langle J\rangle _\bbQ$. Let $g\in Hdg _A(\bbR)$, and consider $J':=gJg^{-1}$ as another complex structure on $V$, so that we have the complex torus $A':=(V/\Gamma ,J')$.

(1) Then $A'$ is also an abelian variety and we have the inclusions of Neron-Severi groups $NS_A\subset NS_{A'}$ and of ample cones $C^a_A\subset C^a_{A'}$.

(2) Assume in addition that $\langle J'\rangle _\bbQ=Hdg _A$, i.e. $Hdg _A=Hdg _{A'}$.  Then $C_A^a=C^a_{A'}$.
\end{lemma}

\begin{proof} (1) By definition an integral skew symmetric form $s$ on $V$ belongs to $NS_A$ if $s$ is $J$-invariant. This happens if and only if $s$ is invariant under all elements of $Hdg _A(\bbR)$.
We have
$$Hdg _{A'}=\langle J^\prime \rangle _\bbQ\subset Hdg_A$$
Hence $NS_A\subset NS_{A'}$.

A skew-symmetric form $s\in NS_A$ represents an ample class if and only if the quadratic form $s(J(-),-)$ is positive definite on $V$. In this case for $0\neq x\in V$ we have
$$s(J'x,x)=s(gJg^{-1}x,x)=s(Jg^{-1}x,g^{-1}x)>0$$
Thus $s$ represents an ample class in $A'$, so $A'$ is an abelian variety and also $C_A^a\subset C_{A'}^a$.

(2) follows from (1), because we can interchange $A$ and $A'$.
\end{proof}

\begin{cor} \label{indep of ample cone} (1) For any parameter $\omega \in C_A$ we have $C_B\subset C_{B_\omega}$.

(2) For all $\omega \in C_A^0$ there is the equality $C_B= C_{B_\omega}$
\end{cor}

\begin{proof} (1) Let $\omega \in C_A$. By Lemma \ref{lemma single conj class} the operators $I_\omega $ and $I_{\omega _A}$ are conjugate in the Lie group $U(A)_\bbQ(\bbR)^0=Hdg _B$. So by Lemma \ref{equality of ample cones} $(NS_B=)\ NS_{B_{\omega _A}}\subset NS_{B_\omega}$ and similarly for the ample cones. Therefore also $C_B\subset C_{B_\omega}$.

(2) If in addition $\omega \in C^0_A$, then
by definition $\langle I_\omega \rangle _{\bbQ}=\langle I_{\omega _A}\rangle _{\bbQ}$  and hence $Hdg _{B_\omega }=Hdg _B$.
So it remains to apply Lemma \ref{lemma single conj class} and   Lemma \ref{equality of ample cones}.
\end{proof}



\subsection{Proof of Conjecture \ref{subs main conj} for mirror families of abelian varieties.}\label{proof of conj for fam of ab var}  By Corollary \ref{cor coincide} we have $U(A_\omega)=U(A)$ for $\omega \in C_B^0$ and by Remark \ref{remark commensur}, for any abelian variety $C$, the groups $G^{eq}(C)$ and $U(C)$ are isomorphic up to finite groups. We conclude that for the family $\cA$ the group $G^{eq}(\cA)$ (\ref{subs main conj}) is isomorphic up to finite groups to $U(A)$ and also to $U(A)^0$ (Remark \ref{remark commeas}).

Now consider the mirror dual family
$$f:\cB \to C_A$$
We need to determine the monodromy group (Definition \ref{def mon}) of this family and to prove that it is isomorphic up to finite groups to $U(A)^0$. This will require a few steps.

Recall that by our assumption we have the equality of algebraic $\bbQ$-subgroups of $SO(\Lambda _\bbQ, Q_\bbQ)$: $U(A)_\bbQ =Hdg _B$ \eqref{equality of mirror groups}. Hence in particular the group $U(A)_\bbQ$ preserves the subspace $\Gamma _{B,\bbQ}\subset \Lambda _\bbQ$ and the groups $U(A)_\bbQ$ and $Hdg _B$ can (and will) be considered as subgroups of $Gl(\Gamma _{B,\bbQ})$.

\begin{defn} Let $G\subset Gl(\Gamma _B)$ be the set of elements $g$ for which there exist $\omega _1,\omega _2\in C^0_A$ (that may depend on $g$) such that
\begin{equation}\label{commut with cx str}  gI_{\omega _1}g^{-1}=I_{\omega _2}
\end{equation}
\end{defn}

\begin{prop} \label{basic} The following holds for the set $G$:

(1) For each $g\in G$ we have $g U(A)_\bbQ(\bbR)^0 g^{-1}=U(A)_\bbQ(\bbR)^0$ and $g U(A)^0 g^{-1}=U(A)^0$;

(2) The conjugation action of $g$ on $U(A)_\bbQ(\bbR)^0$ preserves the conjugacy class $\{I_\omega \ \vert \ \omega \in C_A\}$. In particular $G$ is a subgroup of $Gl(\Gamma _{B,\bbQ})$;

(3) We have the inclusion of groups $U(A)^0\subset G$.

(4) $G$ acts on the space $C_A$ and this action lifts to an action on the family of abelian varieties $f:\cB \to C_A $;

(5) The quotient $C_A^0/G$ is the coarse moduli space of abelian varieties which appear in the family $\cB \vert _{C_A^0}$.
\end{prop}

\begin{proof} (1) Let $g\in G$, and let $\omega _1,\omega _2\in C_A^0$ such that $g I_{\omega _1} g^{-1}=I_{\omega _2}$. Since $U(A)_\bbQ =\langle I_{\omega _1}\rangle _{\bbQ} =\langle I_{\omega _2}\rangle _{\bbQ}$ and $g$ is an integral matrix it follows that
$g U(A)_\bbQ  g^{-1}=U(A)_\bbQ $. Hence also
$g U(A)_\bbQ (\bbR)^0 g^{-1}=U(A)_\bbQ(\bbR)^0$. The group $U(A)^0$ is the subgroup of integral points in $U(A)_\bbQ(\bbR)^0$. Hence also $g U(A)^0 g^{-1}=U(A)^0$.

(2) The first assertion follows from Lemma \ref{lemma single conj class}. For the second one notice that $G$ consists of integral matrices, hence the action of $g\in G$ on the conjugacy class $\{I_\omega \ \vert \ \omega \in C_A\}$ preserves the subset $\{I_\omega \ \vert \ \omega \in C_A^0\}$.So $G$ is a subgroup of $Gl(\Gamma _B)$.

(3) Since $U(A)^0\subset Gl(\Gamma _B)$ consists of integer matrices, its adjoint action on the conjugacy class $\{I_\omega \ \vert \ \omega \in C_A\}$ preserves the subset $\{I_\omega \ \vert \ \omega \in C_A^0\}$. Therefore $U(A)^0\subset G$.

(4) Recall that the correspondence $\omega \mapsto I_\omega$ is injective, which means that $G$ acts on $C_A$. If for $g\in G$ and $\omega _1,\omega _2\in C_A$ we have $gI_{\omega _1}g^{-1}=I_{\omega _2}$ then $g:\Gamma _B\to \Gamma _B$ defines an isomorphism of abelian varieties $g:B_{\omega _1}\stackrel{\sim}{\to} B_{\omega _2}$. Therefore the $G$ action on $C_A$ lifts to an action on the family $\cB$.

(5) Assume that for $\omega _1,\omega _2\in C_A^0$ the abelian varieties $B_{\omega _1}$ and $B_{\omega _2}$ are isomorphic. Then there exists an element $g\in Gl(\Gamma _B)$ such that
$gI_{\omega _1}g^{-1}=I_{\omega _2}$. By definition we have $g\in G$, i.e. $\omega _1,\omega _2$ lie in the same orbit of the group $G$.
\end{proof}

\begin{cor} \label{cor on ex seq} (1) The group $G$ acts by automorphisms of the Lie group $U(A)_\bbQ (\bbR)^0$ and we have the exact sequence of groups
\begin{equation}\label{exact sec for action}
1\to Aut(B)\to G\to Aut(U(A)_\bbQ(\bbR))
\end{equation}

(2) The monodromy group $G^{mon}(\cB)$ of the family $\cB$ is $G/Aut(B)$ (Definition \ref{def mon}).
\end{cor}

\begin{proof} (1). By Proposition \ref{basic} the group $H$ acts on the Lie group $U(A)_\bbQ(\bbR)^0=Hgd _B(\bbR)^0$ by conjugation. The kernel of this action consists precisely of operators $g\in Gl(\Gamma _B)$ such that $g_\bbR$ commutes with the complex structure $J_B$ on $\Gamma _{B,\bbR}$, i.e. $g$ is an automorphism of the abelian variety $B$.

(2) Obviously ${\bf R}^\bullet f_*\bbQ _{\cB }$ is the trivial local system on $C_A$ with fiber $\wedge ^\bullet \Gamma _{B,\bbQ}$. The discrete group $G$ acts on this family and its action on the space of global sections of this local system is clearly effective.
By (1) the kernel of the $G$-action on $C_A$ is $Aut(B)$.
Moreover by Proposition \ref{basic} the quotient space $C_A^0/G$ is the coarse moduli space of abelian varieties in the family $\cB \vert_{C^0_A}$.
Recall that $C^0_A$ is the complement of a countably many analytic subsets in $C_A$. It remains to show that the $G/Aut(B)$-action on $C_A$ is generically free, i.e. it is free outside a countable number of analytic subsets.

The group $G/Aut(B)$ acts effectively on the space $C_A$ which is a conjugacy class in $U(A)_\bbQ(\bbR)^0$, hence an algebraic variety. The fixed subset $C_A^g\subset C_A$ for any $1\neq g\in G/Aut(B)$ is a proper algebraic subvariety. Hence the set
$$Z:=\bigcup _{1\neq g\in G/Aut(B)}C^g_A$$
is a union of countably many proper analytic subsets and the $G/Aut(B)$-action on its complement $C_A\backslash Z$ is free. (Note that $C_A\neq Z$, since $C_A$ is a Baire space.) So the action on $C^0_A\backslash (Z\cap C^0_A)$ is also free.
Thus according to Definition \ref{def mon} the group $G/Aut(B)$ is the monodromy group of the family $\cB$.
\end{proof}

It remains to prove the following

\begin{prop} \label{prop commeas} The group $G/Aut(B)$ is isomorphic up to finite groups to the group $U(A)^0$.
\end{prop}

\begin{proof} Denote the group $G/Aut(B)$ by $H$. By Corollary
\ref{cor on ex seq}, $H\subset Aut(U(A)_\bbQ(\bbR)^0)$. Since the Lie group $U(A)_\bbQ(\bbR)^0$ is semi-simple, the group of inner automorphisms $Inn(U(A)_\bbQ(\bbR)^0)$ has finite index in $Aut(U(A)_\bbQ(\bbR)^0)$. Put $H':=H\cap Inn(U(A)_\bbQ(\bbR)^0)$. Again semi-simplicity of $U(A)_\bbQ(\bbR)^0$ implies that its center
is finite. Denote by $H''$ the preimage of $H'$ in $U(A)_\bbQ(\bbR)^0$ under the conjugation action homomorphism
$$Ad:
U(A)_\bbQ(\bbR)^0\to Inn(U(A)_\bbQ(\bbR)^0)$$
The groups $H''$ and $H$ are isomorphic up to finite groups, so it suffices to show that $H''$ and $U(A)^0$ are isomorphic up to finite groups.

Note that the inclusion $U(A)^0\subset G$ (Proposition \ref{basic}) induces the inclusion $U(A)^0\subset H''$ and it suffices to prove that $U(A)^0$ is a subgroup of finite index in $H''$. The conjugation action of $H''$ on $U(A)_\bbQ(\bbR)^0$ preserves its arithmetic subgroup $U(A)^0$.

The assertion now follows from the general lemma.

\begin{lemma} \label{furman} Let $G$ be an algebraic $\bbQ$-group such that  $G(\bbR)$ is a semisimple Lie group without compact factors and let $L\subset G(\bbR)$ be an arithmetic subgroup. Then $L$ has finite index in its normalizer $N:=N_{G(\bbR)}(L)$.
\end{lemma}

\begin{proof} First notice that by Borel's theorem \cite[Thm.2]{Bor} the normalizer $N$ is contained in the group $G(\bbQ)$ of rational points of $G$. But $N$ is a closed subgroup of $G(\bbR)$, so $N$ is discrete.

By a theorem of Borel and Harish-Chandra  \cite[Thm.1]{BorHC} the arithmetic subgroup $L\subset G(\bbR)$ is a lattice, i.e. the homogeneous space $G(\bbR)/L$ has finite volume.


Let $1\in U\subset G(\bbR)$ be a neighborhood of identity with $U\cap N=1$, and let $V$ be a symmetric
neighborhood of $1$ with $V^2\subset U$. The subsets $\{nV \}_{n\in N}$ are disjoint sets of the same
positive Haar measure and there are $[N:L]$-many of them that project injectively into $G(\bbR)/L$.
The latter has finite volume, so $N/L$ is finite. This proves the lemma and completes the proof of Proposition \ref{prop commeas}.
\end{proof}
\end{proof}

Summarizing the discussion we can formulate the final result.

\begin{thm} Conjecture \ref{subs main conj} holds for mirror families of abelian varieties.
\end{thm}


\subsection{An example of mirror symmetric families of abelian varieties} We will construct an example of perfectly mirror symmetric algebraic pairs (which then by construction in \ref{construction} gives rise to mirror symmetric families of abelian varieties).

Start with a lattice $\Gamma =\oplus _{i=1}^{2n}\bbZ e_i$ and a skew symmetric form $s$ on $\Gamma$ defined by
$$s(e_i,e_{j+n})=\delta _{ij}\quad \text{for $1\leq i,j\leq n$}$$
Consider the operator $J_0\in Gl(\Gamma)$
$$J_0(e_i)=-e_{i+n},\quad J_0(e_{i+n})=e_i\quad \text{for $1\leq i\leq n$}$$
Then $J_0^2=-1$. Moreover, $J_0$ preserves the bilinear form and for each $0\neq \gamma \in \Gamma $ we have $s(J_0\gamma ,\gamma )>0$.
That is the symmetric bilinear form $b(v,w):=s_{\bbR}(J_0v,w)$ is  positive definite on $V:=\Gamma _{\bbR}$. Therefore $J_0$ is a complex structure on the $\bbR$-vector space $V$ and $A_0=(V/\Gamma ,J_0)$ is an abelian variety with polarization $s$.

Consider the symplectic Lie group $Sp (V,s_{\bbR})$. We have $J_0\in Sp (V,s_{\bbR})$ and the centralizer of $J_0$ is the maximal compact subgroup
$$K_{J_0}:=\{g\in Sp (V,s_{\bbR})\ \vert \ b(gv,gw)=b(v,w)\quad \text{for all $v,w\in V$}\}$$
The operator $AdJ_0$ is a Cartan involution on the Lie algebra $Lie (Sp (V,s_{\bbR}))$ corresponding to the subgroup $K_{J_0}$. Let $C(J_0)\subset Sp (V,s_{\bbR})$ be the conjugacy class of $J_0$.
We get a family of abelian varieties $\{A_J:=(V/\Gamma ,J)\ \vert \ J\in C(J_0)\}$ with polarization $s$.
As in the proof of Proposition \ref{generic perfect} one can show  that for a general $J\in C(J_0)$
\begin{equation}\label{generic}
\langle J\rangle _\bbQ=Sp(\Gamma _\bbQ,s_\bbQ)=Sp_{2n,\bbQ}
\end{equation}
Fix $J\in C(J_0)$ such that \ref{generic} holds. Then the corresponding abelian variety $A=(V/\Gamma ,J)$ has $NS_{A}=\bbZ s$. Moreover, $\End (A)=\End (\hat{A})=\bbZ$ and $\Hom (A,\hat{A})=\bbZ s$, $\Hom (\hat{A},A)=\bbZ s^{-1}$. Therefore the group $U(A)$ is isomorphic to $SL_2(\bbZ)$, and $$\U _{A,\bbQ}=U(A)_\bbQ =SL_{2,\bbQ}$$
As usual we may consider the group $Hdg _A=Sp_{2n,\bbQ}$ as a subgroup of the special orthogonal group $SO(\Lambda _{A,\bbQ})$. Then it is easy to see that $Hdg _A$ is the centralizer of $U(A)_\bbQ$ in $SO(\Lambda _{A,\bbQ})$. Hence $Hdg _A$ and
$U(A)_\bbQ$ are mutual centralizers in $SO(\Lambda _{A,\bbQ})$.

Since $NS_A=\bbZ$, for any $\omega _A\in C_A$ the algebraic pair $(A,\omega _A)$ has a mirror symmetric pair $(B,\omega _B)$ (moreover in our case $B$ is a power of an elliptic curve)  \cite[9.6.3]{GLO}. As usual we may assume that $\Lambda _A=\Lambda _B =\Lambda $. Let us choose $\omega _A$ so that the pair $(A,\omega _A)$ is perfect, i.e. $\langle I_{\omega _A}\rangle _\bbQ=U(A)_\bbQ $ (Proposition \ref{generic perfect}). Then we claim that the corresponding pair $(B,\omega _B)$ is also perfect. Indeed,
by definition
$$\langle I_{\omega _B}\rangle _\bbQ =\langle J\rangle _\bbQ=Hdg _A=Sp_{2n,\bbQ}$$
which is the centralizer of
$$U(A)_\bbQ=\langle I_{\omega _A}\rangle _\bbQ =Hdg _B$$
in $SO(\Lambda _\bbQ)$.
But $U(B)_\bbQ$ is contained in the centralizer of $Hdg _B$, hence
$\langle I_{\omega _B}\rangle _\bbQ=U(B)_\bbQ$, i.e. the pair $(B,\omega _B)$ is perfect.

We have constructed perfectly mirror symmetric pairs of abelian varieties $(A,\omega _A)$ and $(B,\omega _B)$, which give rise to a mirror symmetric pairs of abelian varieties as in
\ref{construction}.

\section{Towards Conjecture \ref{subs main conj} for Calabi-Yau hypersurfaces in dual toric varieties}\label{CY}

\subsection{} Many examples of mirror symmetric families of Calabi-Yau varieties were constructed by Batyrev \cite{Bat}.
He starts with two dual lattices $M\simeq \bbZ ^{n+1}$ and $N=M^*$ and a pair of dual reflexive polytopes $\Delta \subset M_\bbQ$, $\Delta ^*\subset N_\bbQ$. These polytopes define a pair of
dual projective Gorenstein toric Fano varieties $\bbP _\Delta$ and $\bbP _{\Delta ^*}$. The corresponding families $\cY$ and $\cY^*$ of anticanonical Gorenstein Calabi-Yau divisors in $\bbP _\Delta $ and $\bbP _{\Delta ^*}$ are expected to be {\it mirror symmetric} \cite{Bat} (one takes the family of all anticanonical divisors and removes the ones that are not Gorenstein) .

From our perspective the problem here is that both families $\cY$ and $\cY ^*$ may consist of singular Calabi-Yau varieties,  in which case we do not want to consider their derived categories.
We can only test our Conjecture \ref{subs main conj} in case the general member of the family is smooth.

It is proved in \cite{Bat} that there always exist {\it maximal projective crepant partial (MPCP)} toric resolutions
$\hat{\bbP}_\Delta \to \bbP _\Delta$ and $\hat{\bbP}_{\Delta ^*} \to \bbP _{\Delta ^*}$. The projective toric varieties $\hat{\bbP}_\Delta$ and $\hat{\bbP}_{\Delta ^*}$ correspond to simplicial fans (in $N_\bbQ$ and $M_\bbQ$ respectively), hence they have only quotient singularities. Moreover the pullbacks $\cX$ and $\cX ^*$ of families $\cY$ and $\cY ^*$ will consist of projective
$n$-dimensional Gorenstein Calabi-Yau varieties which are {\it quasi-smooth}, i.e. have only quotient singularities. In particular, for members $X$ and $X^*$ of these families the rational cohomology spaces $H^\bullet (X,\bbQ)$ and $H^\bullet (X^*,\bbQ)$ satisfy Poicare duality and possess pure Hodge structure.

The main result of \cite{BatBor} implies the mirror symmetry for the corresponding Hodge numbers
\begin{equation}\label{mir sym for hodge diamonds} h^{p,q}(X)=h^{n-p,q}(X^*), \quad 0\leq p,q\leq n
\end{equation}

It might happen that varieties $X$ and $X^*$ are smooth
(this is always so if $n=3$ \cite{Bat})
in which case one can test Conjecture \ref{subs main conj} for the families $\cX$ and $\cX ^*$ and
we believe that it holds. Unfortunately we are unable to prove this in full generality. Instead we have some partial results in the direction of the conjecture. Let us make two technical assumptions.

\medskip

\noindent{\it Assumption A.} We assume that the polytope $\Delta $ is integral, i.e. $\Delta \subset M$, and that the projective Gorenstein toric Fano variety $\bbP _\Delta$ is {\it smooth}.

\medskip

This implies that $\hat{\bbP}_\Delta =\bbP _\Delta$, $\cX$ is the family of (very ample) smooth anticanonical divisors in $\bbP _\Delta$. So all members $X$ of the family $\cX$ are smooth projective Calabi-Yau varieties.

\medskip

\noindent{\it Assumption B.} We assume that $n=\dim X$ is {\it odd}.

\medskip

The two assumptions imply that the Hodge diamond of $X$ is a cross:
\begin{equation}\label{cross} H^\bullet (X,\bbQ)=H^n(X,\bbQ)\oplus (\oplus _pH^{p,p}(X,\bbQ))
\end{equation}
Indeed, $X$ is a hyperplane section of a smooth projective toric variety $\bbP _\Delta$ whose cohomology consists of algebraic cycles. It remains to apply the weak Lefschetz theorem to the pair $X\subset \bbP _\Delta$.

The group $G^{eq}(X)$ preserves the Mukai pairing on the even cohomology $H^{even}(X,\bbQ)=\oplus _pH^{p,p}(X,\bbQ)$. Because $c_1(X)=0$ and $\dim X$ is odd, this Mukai pairing is skew-symmetric \cite[Exercise 5.43]{Huyb}. Therefore $G^{eq}(X)\subset Sp(H^{even}(X,\bbQ))$.

On the mirror side we have no reason to believe that the general member $X^*$ of the dual family $\cX ^*$ is smooth. However the relation \eqref{mir sym for hodge diamonds} implies that
\begin{equation}\label{equality of dimensions}
\dim H^{even}(X,\bbQ)=\dim H^n(X^*,\bbQ) \quad \text{and}\quad \dim H^{even}(X^*,\bbQ)=\dim  H^n(X,\bbQ)
\end{equation}

The monodromy group of $G^{mon}(\cX^*)$  acts trivially on the even cohomology $H^{even}(X^*,\bbQ)$ and preserves the Poincare pairing on the middle cohomology $H^n(X^*,\bbQ)$. Since $n$ is odd this pairing is skew-symmetric and therefore $G^{mon}(\cX^*)\subset Sp(H^n(X^*,\bbQ))$.

We find that the discrete groups in question, $G^{eq}(X)$ and
$G^{mon}(\cX^*)$, are contained in isomorphic symplectic groups.

\begin{remark} (1) Notice that in fact $G^{mon}(\cX^*)$ is contained in $Sp(H^n(X^*,\bbZ))$ and we expect that it is a subgroup of finite (small) index. For the family of smooth hypersurfaces in a projective space this is a theorem of Beauville \cite{Beau}.

(2) On the other hand the action of $G^{eq}(X)=G^{eq}(\cX)$ on $H^{even}(X,\bbQ)$ does not in general preserve the lattice $H^{even}(X,\bbZ)$. However it preserves a different lattice, which is the image of the topological K-theory $K^0_{top}(X)$ under the Mukai vector isomorphism $v: K^0_{top}(X)\otimes \bbQ\to H^{even}(X,\bbQ)$
\cite{AdTho}.

{\bfseries We expect} the groups $G^{mon}(\cX ^*)$ and $G^{eq}(\cX)$ to be arithmetic subgroups in the corresponding isomorphic symplectic groups $Sp(H^n(X^*,\bbQ))$ and $Sp(H^{even}(X,\bbQ)$ (which would prove Conjecture \ref{subs main conj} for families $\cX$ and $\cX ^*$).
\end{remark}

\subsection{} The next theorem is an indication that $G^{eq}(X)$ may indeed be an arithmetic subgroup of $Sp (H^{even}(X,\bbQ))$.

\begin{thm} \label{half about geq} For every member $X$ of the family $\cX$ the discrete group $G^{eq}(X)$ is Zariski dense in $Sp (H^{even}(X,\bbQ))$.
\end{thm}

\begin{proof} As explained above, the Mukai pairing on $H^{even}(X,\bbQ)$ is skew-symmetric and $G^{eq}(X)\subset Sp(H^{even}(X,\bbQ))$. Let $\overline{G^{eq}(X)}\subset Sp(H^{even}(X,\bbQ))$ be the algebraic $\bbQ$-subgroup which is the Zariski closure of $G^{eq}(X)$. To prove the equality $\overline{G^{eq}(X)}= Sp(H^{even}(X,\bbQ))$
it suffices to show the equality of the Lie groups $\overline{G^{eq}(X)}(\bbC)=Sp(H^{even}(X,\bbC))$.

Since the smooth projective variety $X$ is Calabi-Yau, every line bundle $L$ on $X$ is a spherical object in $D^b(X)$ and as such it defines the corresponding spherical twist functor $T _L$ \cite[Def. 8.3]{Huyb} which is an autoequivalence of the derived category $D^b(X)$. For any spherical object $E\in D^b(X)$ the action of the corresponding twist functor on the cohomology $H^\bullet (X,\bbQ)$ is  the reflection with respect to the Mukai vector $v(E)$:
$$r_{v(E)}(x):=x-\langle v(E),x\rangle v(E)$$
where $\langle -,-\rangle$ is the Mukai pairing on $H^\bullet (X,\bbQ)$ \cite[8.12]{Huyb}. Let
$$Q=\{\delta \in H^{even}(X,\bbC)\ \vert \ r_\delta \in \overline{G^{eq}(X)}(\bbC)\}$$
Note that $Q$ is a closed subset in $H^{even}(X,\bbC)$ and $v(M)\in Q$ for all line bundles $M\in Pic(X)$. If $\delta \in Q$ and $g\in \overline{G^{eq}(X)}(\bbC)$, then also
$$r_{g(\delta)}=g\cdot r_\delta \cdot g^{-1}\in \overline{G^{eq}(X)}(\bbC)$$
So $Q$ is $\overline{G^{eq}(X)}(\bbC)$-invariant.

Moreover, if $\delta \in Q$, then for $k\in \bbZ$
\begin{equation}\label{inf order} r^k _{\delta }(x)=x -k\langle \delta ,x\rangle \delta \quad \text{for all $x\in H^{even}(X,\bbC)$}
\end{equation}
It follows that for any $\delta \in Q$ the 1-parameter subgroup
$$U_\delta =\{x\mapsto x+\lambda \langle x,\delta \rangle \delta \ \vert \ \lambda \in \bbC\}$$
belongs to $\overline{G^{eq}(X)}(\bbC)$. In particular the whole line spanned by $\delta $ is in $Q$. So $Q$ is a cone over the origin in $H^{even}(X,\bbC)$.

At this point we recall the following lemma of Deligne.

\begin{prop} \label{deligne's lemma}  Let $(V,\psi )$ be a finite dimensional symplectic $\bbC$-vector space, $G\subset Sp(V,\psi)$ an algebraic subgroup. Let $R\subset V$ be an $G$-orbit, which spans $V$. Assume that for every $\delta \in R$, $G$ contains the 1-parameter subgroup
$U_\delta =\{x\mapsto x+\lambda (x,\delta)\delta \ \vert \ \lambda \in \bbC\}$. Then $G=Sp(V,\psi)$.
\end{prop}

\begin{proof} This is \cite[Lemma 4.4.2]{Del3}
\end{proof}

To apply this proposition to our case $V=H^{even}(X,\bbC)$, $G=\overline{G^{eq}(X)}(\bbC)$ it suffices to find an element of $Q$, whose $\overline{G^{eq}(X)}(\bbC)$-orbit spans
$H^{even}(X,\bbC)$.  We will show that the fundamental class $\eta \in H^{2n}(X,\bbC)$ is such an element.
For this we will analyze Mukai vectors $v(L)$ of line bundles and their Mukai pairing $\langle v(L_1),v(L_2)\rangle$.

For $a\in H^{even}(X,\bbQ)$ we denote its $i$-th homogeneous component by $a_i$.

\begin{lemma}\label{spans}  The Mukai vectors $v(L)$ of line bundles $L\in Pic(X)$ span the vector space $H^{even}(X,\bbQ)$.
\end{lemma}

\begin{proof} First notice that the ring $H^{even}(X,\bbQ)$ is generated by $H^2(X,\bbQ)=NS(X)_\bbQ$. Indeed, it is well known that the cohomology ring $H^\bullet (\bbP _\Delta ,\bbQ)$ of the nonsingular projective toric variety $\bbP _\Delta$ is generated by $H^2(\bbP _\Delta ,\bbQ)=NS(\bbP _\Delta )_\bbQ$. The smooth subvariety $X\subset \bbP _\Delta$ is a hyperplane section, so by the weak Lefschetz theorem the part $H^{<n}(X,\bbQ)$ is generated by $H^2(X,\bbQ)=NS(X)_\bbQ$. Now the hard Lefschetz theorem for $X$ implies that the whole ring $H^\bullet (X,\bbQ)$ is generated by $H^2(X,\bbQ)$.

To prove the lemma recall that
$$v(F)=ch(F)\cup \sqrt{td_X}$$
where  $(\sqrt{td_X})_0=1$ and hence $\sqrt{td_X}$ is invertible in the ring $H^{even}(X,\bbQ)$. So it suffices to show that the Chern characters of line bundles span $H^{even}(X,\bbQ)$. Let $L_1,...,L_p$ be line bundles such that $c_1(L_1),...,c_1(L_p)$ form a basis of $H^2(X,\bbQ)$. Put $x_i:=c_1(L_i)$. Then monomials $x_1^{m_1}\cdot ...\cdot x_p^{m_p}$ span $H^{even}(X,\bbQ)$ and hence the Chern characters
$$\{ ch(L_1^{k_1}\otimes ...\otimes L_p^{k_p})=\sum _{m_1,...,m_p\geq 0}
\frac{k_1^{m_1}\cdot ...\cdot k_p^{m_p}}{m_1!\cdot ...\cdot m_p!}x_1^{m_1}\cdot ...\cdot x_p^{m_p} \mid  k_1,...,k_p\in \bbZ \}$$
span $H^{even}(X,\bbQ)$ as well.
\end{proof}

Let $\eta \in H^{2n}(X,\bbQ)$ be the fundamental class. Fix an ample line bundle $L$. Then the top component
$(v(L^m))_{2n}$ of the Mukai vector $v(L^m)$ as a function of $m$ is
$$am^n\eta+\text{lower terms},\quad \text{with $a>0$}$$
whereas the components $(v(L^m))_{2d}$ with $d<n$ grow no faster than $m^d$. So as $m\to \infty$ the lines spanned by the Mukai vectors $v(L^m)$ will tend to the line $H^{2n}(X,\bbC)$.
Since $Q$ is a closed subset of $H^{even}(X,\bbC)$ which is a cone over the origin, we find that the line $H^{2n}(X,\bbC)\subset Q$. Therefore $\eta \in Q$.

Note that for any line bundle $M\in Pic(X)$
$$r_{v(M)}(\eta)=\eta -\langle v(M),\eta \rangle v(M)=\eta -v(M)$$
\cite[5.42]{Huyb}. Let $M_1,...,M_t$ be line bundles whose Mukai vectors span $H^{even}(X,\bbC)$. Then the vectors $\eta$ and $r_{v(M_i)}(\eta)$ span $H^{even}(X,\bbC)$ and belong to an orbit of
$\overline{G^{eq}(X)}(\bbC)$, which completes the proof of Theorem \ref{half about geq}.
\end{proof}

\subsection{} On the mirror side we have no definite results. However let us recall the situation with the universal family of $d$-dimensional hypersurfaces $\cY \to S$ in a projective space $\bbP ^{d+1}$.
Assume that $d$ is odd. Let $S^0\subset S$ be the subset parametrizing smooth hypersurfaces, $s\in S^0$ and $Y_s$ the corresponding smooth hypersurface. We are interested in the monodromy representation of $\pi _1(S^0,s)$ in the middle cohomology
$H^d(Y_s)$. A theorem of Beauville \cite{Beau} asserts that the monodromy group is a subgroup of finite index in the corresponding arithmetic group $Sp(H^d(Y_s,\bbZ))$. The proof of this theorem uses a trick and relies on earlier results of \cite{Jan}. However it is relatively easy to prove that the monodromy group is Zariski dense in $Sp(H^d(X,\bbQ))$. Recall the relevant well known facts \cite{Voi}.

(1) For most projective lines $\Delta \simeq \bbP ^1\subset S$ the restriction of the universal family $\cY \to S$ to $\Delta$ is a Lefschetz pencil. That is there exist a finite number of critical values $t_1,...,t_n\in \Delta$ and the corresponding singular fibers $Y_i$ have a unique nondegenerate singular point.

(2) The map of the fundamental groups $\pi _1(\Delta \cap S^0,s)\to \pi _1(S^0,s)$ is surjective.

(3) For each $t_i$ there exists a unique {\it vanishing cycle} $\delta _i\in H^d(Y_s,\bbQ)$ with the following properties:
\begin{itemize}
\item The cycles $\delta _i$, $i=1,...,n$ span $H^d(Y_s,\bbQ)$.
\item The local monodromy around $t_i$ is the reflection about $\delta _i$, i.e. it is $x\mapsto x\pm (x,\delta _i)\delta _i$, where $(-,-)$ is the skew-symmetric Poincare pairing on $H^d(Y_s,\bbQ)$.
\item The monodromy representation of $\pi _1(\Delta \cap S^0,s)$ in $H^d(Y_s,\bbQ)$ is irreducible.
    \end{itemize}

It is not difficult to deduce from (3) that the monodromy group is Zariski dense in $Sp(H^d(Y_s,\bbQ))$ \cite[5.11]{Del2}.

We recalled the case of hypersurfaces in the projective space to stress the analogy between the monodromy group $G^{mon}$ and the group $G^{eq}$ as in Theorem \ref{half about geq}. Indeed, both groups contain "many reflections".

Coming back to our family $\cX ^*$ of quasi-smooth Calabi-Yau varieties in $\hat{\bbP}_{\Delta ^*}$, we don't know if Lefschetz pencils with the properties (1),(2),(3) exist, and so we do not know how to analyze the monodromy group $G^{mon}(\cX ^*)$.

\end{document}